\documentclass[11pt]{amsart} \usepackage{latexsym, amssymb}
\usepackage{amssymb,latexsym,amsmath,amsthm,amscd}

\newtheorem{thm}{Theorem}[section]
\newtheorem{lemma}[thm]{Lemma}

\newtheorem{cor}[thm]{Corollary}

\theoremstyle{definition}
\newtheorem{df}[thm]{Definition}
\newtheorem{nrmk}[thm]{Remark}
\newtheorem{ex}[thm]{Example}

\theoremstyle{remark}
\newtheorem*{rmk}{Remark}

\renewcommand{\r}{\mathbb{R}}
\newcommand{\n}{\mathbb{N}}
\newcommand{\z}{\mathbb{Z}}
\newcommand{\x}{\mathbb{X}}
\newcommand{\y}{\mathbb{Y}}
\renewcommand{\O}{\mathbb{O}}
\renewcommand{\k}{\mathcal{K}}
\renewcommand{\u}{\mathcal{U}}
\newcommand{\w}{\mathcal{W}}
\renewcommand{\o}{\Omega}
\renewcommand{\t}{\Lambda}
\newcommand{\muu}{\boldsymbol{\mu}}
\def \a{\operatorname{Ad}}
\def \st{\operatorname{st}}
\def \Ns{\operatorname{NSS}}
\def \Nsc{\operatorname{NSCS}}
\def \ord{\operatorname{ord}}
\def \ns{\operatorname{ns}}
\def \supp {\operatorname{supp}}
\def \bigO{\operatorname{O}}
\def \lilo{\operatorname{o}}
\def \Aut{\operatorname{Aut}}
\def \dom{\operatorname{domain}}
\def \im{\operatorname{image}}
\def \dim{\operatorname{dim}}
\def \interior{\operatorname{interior}}
\def \GL{\operatorname{GL}}
\def \id{\operatorname{id}}
\def \C{\operatorname{C}}

\begin{document}

\title{Hilbert's Fifth Problem for Local Groups}
\author{Isaac Goldbring}
\address {University of Illinois, Department of Mathematics, 1409 W. 
Green street, Urbana, IL 61801}

\email{igoldbr2@math.uiuc.edu}
\urladdr{www.math.uiuc.edu/~igoldbr2}

\begin{abstract}
We solve Hilbert's fifth problem for local groups:  every locally euclidean local group is locally isomorphic to a Lie group.  Jacoby claimed a proof of this in 1957, but this proof is seriously flawed.  We use methods from nonstandard analysis and model our solution after a treatment of Hilbert's fifth problem for global groups by Hirschfeld.
\end{abstract}
\date{}
\maketitle

\section{Introduction}

\noindent The most common version of Hilbert's fifth problem asks whether every locally euclidean topological group is a Lie group.  More precisely, if $G$ is a locally euclidean topological group, does there always exist a $C^\omega$ structure on $G$ such that the group operations become $C^\omega$?  In 1952, Gleason, Montgomery, and Zippin answered this question in the affirmative; see ~\cite{G} and \cite{MZ1}.  (``$C^\omega$" means ``real analytic," and by a $C^\omega$ structure on a topological space $X$ we mean a real analytic manifold, everywhere of the same finite dimension, whose underlying topological space is $X$.)  

\

However, Lie groups, as first conceived by Sophus Lie, were not actually groups but rather ``local groups" where one could only multiply elements that were sufficiently near the identity.  (A precise definition of ``local group" is given in the next section.)   It is thus natural to pose a local version of Hilbert's fifth problem (henceforth referred to as the Local H5):  Is every locally euclidean local group locally isomorphic to a Lie group?  In ~\cite{J}, Jacoby claims to have solved the Local H5 affirmatively.  However, as pointed out by Plaut in ~\cite{P}, Jacoby fails to recognize a subtlety in the way the associative law holds in local groups, as we will now explain.

\

In our local groups, if the products $xy, yz, x(yz),$ and $(xy)z$ are all defined, then $x(yz)=(xy)z$.  This condition is called ``local associativity" by Olver in \cite{O}.  A much stronger condition for a local group to satisfy is ``global associativity" in which, given any finite sequence of elements from the local group and two different ways of introducing parentheses in the sequence, if both products thus formed exist, then these two products are in fact equal.  According to \cite{O}, ``Jacoby's Theorem 8, which is stated without proof, essentially makes the false claim that local associativity implies global associativity."   Contrary to this ``Theorem 8", ~\cite{O} has examples of connected local Lie groups that are not globally associative.  Since, according to \cite{O}, ``the rest of Jacoby's paper relies heavily on his incorrect Theorem 8", this error invalidates his proof.  Olver also indicates that one could probably rework Jacoby's paper to extract the truth of the Local H5.  He suggests that this would be a worthwhile endeavor, as, for example, the solution of Hilbert's fifth problem for cancellative semigroups on manifolds heavily relies on the truth of the Local H5; see ~\cite{B} and ~\cite{Ho}.

\

Rather than reworking Jacoby's paper, we have decided to mimic the nonstandard treatment of Hilbert's fifth problem given by Hirschfeld \cite{H}.  Whereas nonstandard methods simplified the solution of Hilbert's fifth problem, such methods are even more natural in dealing with the Local H5.  The nonstandard proof of Hilbert's fifth problem involves ``infinitesimals"--elements in the nonstandard extension of the group infinitely close to the identity; therefore, much of \cite{H} goes through in the local setting, since the infinitesimals form an actual group.  These infinitesimals are used to generate local one-parameter subgroups, as in \cite{S} and \cite{H}.

\

While our solution of the Local H5 is along the lines of the proof given by Hirschfeld, many of the arguments have grown substantially in length due to the care needed in working with local groups.  We assume familiarity with elementary nonstandard analysis; otherwise, consult ~\cite{D} or ~\cite{He}.  We should also mention that McGaffey \cite{ Mc} has a completely different
nonstandard approach to Hilbert's fifth problem for local groups, using the connection
between local Lie groups and finite-dimensional real Lie algebras to reduce to a
problem of approximating locally euclidean local groups in a
certain metric by local Lie groups.  This approach has not yet led to a solution.

\

I would like to thank Lou van den Dries for many helpful suggestions. 

\

\section{Preliminaries}

\

\noindent  In this section, we define the notion of local group and state precisely the Local H5 (in two forms).  We also introduce the fundamental objects in the entire story, the local 1-parameter subgroups of $G$.  After defining the local versions of familiar constructions in group theory, e.g. morphisms between local groups, sublocal groups, and local quotient groups, we prove some elementary properties of local groups that are used heavily throughout the paper.  We end this section by fixing some notation concerning the nonstandard setting.

\

\noindent Throughout, $m$ and $n$ range over $\n:=\{0,1,2,\ldots\}$.

\newpage

\noindent \textbf{Statement of the Local H5}

\

\begin{df} A \textbf{local group} is a tuple $(G,1,\iota,p)$ where $G$ is a hausdorff topological space with a distinguished element $1\in G$,  and $\iota:\t \rightarrow G$ (the inversion map) and $p:\o\rightarrow G$ (the product map) are continuous functions with open $\t \subseteq G$ and open $\o\subseteq G\times G$, such that $1\in \t$, $\{1\}\times G \subseteq \o$, $G \times \{1\} \subseteq \o$, and for all $x,y,z\in G$:

\begin{enumerate}
\item $p(1,x)=p(x,1)=x$;
\item if $x\in \t$, then $(x,\iota(x))\in \o$, $(\iota(x),x)\in \o$ and $$p(x,\iota(x))=p(\iota(x),x)=1;$$
\item if $(x,y),(y,z)\in \o$ and $(p(x,y),z),(x,p(y,z))\in \o$, then 
\begin{equation}
p(p(x,y),z)=p(x,p(y,z)).\tag{A}
\end{equation}
\end{enumerate}
\end{df}

\

\noindent  In the rest of this paper $(G,1,\iota,p)$ is a local group; for simplicity we denote it just by $G$. For the sake of readability, we write $x^{-1}$ instead of $\iota(x)$, and $xy$ or $x\cdot y$ instead of $p(x,y)$.  Note that $\iota(1)=1$ and if $x,y\in \t$, $(x,y)\in \o$, and $xy=1$, then $y=x^{-1}$ and $x=y^{-1}$.

\

\begin{df}

\

\begin{enumerate}
\item Let $U$ be an open neighborhood of $1$ in $G$.  Then the \textbf{restriction of $G$ to $U$} is the local group $G|U:=(U,1,\iota|\t_U,p|\o_U)$, where $$\t_U:=\t \cap U \cap \iota^{-1}(U) \text{ and }\o_U:=\o \cap (U\times U)\cap p^{-1}(U).$$
\item $G$ is \textbf{locally euclidean} if there is an open neighborhood of $1$ homeomorphic to an open subset of $\r^n$ for some $n$.
\item $G$ is a \textbf{local Lie group} if $G$ admits a $C^\omega$ structure such that the maps $\iota$ and $p$ are $C^\omega$.
\item $G$ is \textbf{globalizable} if there is a topological group $H$ and an open neighborhood $U$ of $1_H$ in $H$ such that $G=H|U$.
\end{enumerate}
\end{df}

\

\noindent By a \textit{restriction of $G$} we mean a local group $G|U$ where $U$ is an open neighborhood of $1$ in $G$.

\

\noindent \textbf{Local H5-First Form}:  If $G$ is a locally euclidean local group, then some restriction of $G$ is a local Lie group.

\

\noindent \textbf{Local H5-Second Form}:  If $G$ is a locally euclidean local group, then some restriction of $G$ is globalizable.

\

\noindent The equivalence of the two forms of the Local H5 is seen as follows.  It is known that if $G$ is a local Lie group, then some restriction of $G$ is equal to some restriction of a Lie group.  Thus the first form implies the second form.  Conversely, by the Montgomery-Zippin-Gleason solution to the original H5, if $G|U$ is globalizable, the global group containing $G|U$ will be locally euclidean and there will be a $C^\omega$ structure on it making it a Lie group.  Then $G|U$ will be a local Lie group.

\

\begin{df}
A \textbf{subgroup} of $G$ is a subset $H\subseteq G$ such that $1\in H$, $H\subseteq \t$, $H\times H\subseteq \o$ and for all $x,y\in H$, $x^{-1}\in H$ and $xy\in H$.
\end{df}

\

\noindent Note that then $H$ is an actual group with product map $p|H\times H$.  We will say $G$ has \textbf{no small subgroups}, abbreviated $G$ is $\Ns$, if there is a neighborhood $\mathcal{U}$ of $1$ such that there are no subgroups $H$ of $G$ with $H\not= \{1\}$ and $H\subseteq \mathcal{U}.$  In analogy with the solution to Hilbert's fifth problem in the global case, we will first show that every locally compact NSS local group has a restriction which is a local Lie group and then we will prove that every locally euclidean local group is NSS.

\bigskip

\noindent \textbf{The Sets $\u_n$}

\

\noindent As noted in the introduction, one obstacle in the local theory is the potential lack of generalized associativity.  The next definition captures the informal idea of being able to unambiguously multiply a finite sequence of elements.  

\

\begin{df}
Let $a_1,\ldots,a_n,b\in G$ with $n\geq 1$.  We define the notion $(a_1,\ldots,a_n)$ \textbf{represents} $b$, denoted $(a_1,\ldots,a_n)\rightarrow b$, by induction on $n$ as follows:

\begin{itemize}
\item $(a_1)\rightarrow b$ iff $a_1=b$;
\item $(a_1,\ldots,a_{n+1})\rightarrow b$ iff for every $i\in \{1,\ldots,n\}$, there exists $b_i',b_i''\in G$ such that $(a_1,\ldots,a_i)\rightarrow b_i'$, $(a_{i+1},\ldots,a_{n+1})\rightarrow b_i''$, $(b_i',b_i'')\in \o$ and $b_i'\cdot b_i''=b$.
\end{itemize}
By convention, we say that $(a_1\ldots a_n)$ represents $1$ when $n=0$.
\end{df}

\noindent We will say $a_1\cdots a_n$ is \textbf{defined} if there is $b\in G$ such that $(a_1,\ldots,a_n)$ represents $b$; in this case we will write $a_1\cdots a_n$ for this (necessarily unique) $b$.  If $a_1\cdots a_n$ is defined, then, informally speaking, all possible $n$-fold products are defined and equal.  If $a_1\cdots a_n$ is defined and $a_i=a$ for all $i\in \{1,\ldots,n\}$, we say that $a^n$ is defined and denote the unique product by $a^n$.  In particular, $a^0$ is always defined and equals $1$.  

\

\noindent From now on, put $A^{\times n}:=\underbrace{A\times \cdots \times A}_{n \text{ times}}$.  Also, call $W\subseteq G$ \textbf{symmetric} if $W\subseteq \t$ and $W=\iota(W)$.  Note that $U\cap \iota^{-1}(U)$ is a symmetric open neighborhood of $1$ in $G$ for every open neighborhood $U$ of $1$ contained in $\t$. 

\

\noindent We now prove a lemma which says that we can multiply any number of elements unambiguously provided that the elements are sufficiently close to the identity.  

\

\begin{lemma}
There are open symmetric neighborhoods $\u_n$ of $1$ for $n>0$ such that $\u_{n+1}\subseteq \u_n$ and for all $(a_1,\ldots,a_n)\in \u_n^{\times n}$, $a_1\cdots a_n$ is defined.
\end{lemma}

\begin{proof}
We let $\u_1$ be any open symmetric neighborhood of $1$ and choose $\u_2$ to be a symmetric open neighborhood of $1$ such that $\u_2 \subseteq \u_1$ and $\u_2 \times \u_2\subseteq \o$.  

\

\noindent Assume inductively that $n\geq 2$ and that for $m=1,\ldots,n$, 
\begin{itemize}
\item $\u_m$ is a symmetric open neighborhood of $1$;
\item $\u_{m+1}\subseteq \u_m$ if $m<n$;
\item for all $(a_1,\ldots,a_m)\in \u_m^{\times m}$, $a_1\cdots a_m$ is defined;
\item the map $\varphi_m:\u_m^{\times m}\rightarrow G$ defined by $\varphi_m(a_1,\ldots,a_m)=a_1\cdots a_m$ is continuous.
\end{itemize} 

\noindent Let $\u_{n+1}$ be a symmetric open neighborhood of $1$ such that $\u_{n+1}\subseteq \u_n$ and $\u_{n+1}^{\times n}\subseteq \varphi_n^{-1}(\u_2)$.  We show that this choice of $\u_{n+1}$ works.  Let $(a_1,\ldots,a_{n+1})\in \u_{n+1}^{\times (n+1)}$.  We must show that $(a_1,\ldots,a_{n+1})$ represents $a_1\cdot(a_2\cdots a_{n+1})$.  However, this is immediate by the inductive assumptions, using (A) several times.
\end{proof}

\

\noindent From now on, the sets $\u_n$ will be as in the previous lemma.  Also, if $A\subseteq \u_n$, put $A^n:=\{a_1\cdots a_n \ | \ (a_1,\ldots,a_n)\in A^{\times n}\}$.

\

\noindent \textbf{Local 1-parameter subgroups}

\

\begin{df}
A \textbf{local $1$-parameter subgroup} of $G$, henceforth abbreviated \textbf{local $1$-ps} of $G$, is a continuous map $X:(-r,r)\rightarrow G$, for some $r\in (0,\infty]$, such that 
\begin{enumerate}
\item $\im(X)\subseteq \t$, and 
\item if $r_1,r_2,r_1+r_2\in (-r,r)$, then $(X(r_1),X(r_2))\in \o$ and $$X(r_1+r_2)=X(r_1)\cdot X(r_2).$$
\end{enumerate}
\end{df}

\

\noindent Suppose $X:(-r,r)\rightarrow G$ is a local 1-ps of $G$ and $s\in (-r,r)$.  It is easy to verify, by induction on $n$, that if $ns\in (-r,r)$, then $X(s)^n$ is defined and $X(ns)=X(s)^n$.

\

\noindent We now define the local analog of the space which played a key role in the solution of Hilbert's fifth problem in the global setting.  

\

\begin{df}
Let $X,Y$ be local 1-parameter subgroups of $G$.  We say that $X$ is \textbf{equivalent to} $Y$ if there is $r\in \r^{>0}$ such that $$r\in \dom(X)\cap \dom(Y) \text{ and }X|(-r,r)=Y|(-r,r).$$  We let $[X]$ denote the equivalence class of $X$ with respect to this equivalence relation.  We also let $L(G):=\{[X] \ | \ X \text{ is a local }1\text{-ps of } G\}$.  
\end{df}

\

\noindent In other words, $L(G)$ is the set of germs at $0$ of local 1-parameter subgroups of $G$.  We often write $\x$ for an element of $L(G)$ and write $X\in \x$ to indicate that the local $1$-ps $X$ is a representative of the class $\x$.  We have a \textbf{scalar multiplication map} $(s,\x)\mapsto s\cdot \x:\r \times L(G)\rightarrow L(G)$, where $s \cdot \x$ is defined as follows.  Let $X\in \x$ with $X:(-r,r)\rightarrow G$.  If $s=0$, then $0\cdot \x=\O$, where $O:\r \rightarrow G$ is defined by $O(t)=1$ for all $t\in \r$ and $\O=[O]$.  Else, define $s X:(\frac{-r}{|s |},\frac{r}{|s|})\rightarrow G$ by $(s X)(t)=X(s t)$.  Then $sX$ is a local 1-ps of $G$, and we set $s \cdot \x=[s X]$.  It is easy to verify that for all $\x\in L(G)$ and $s,s' \in \r$, $1\cdot \x=\x$ and $s \cdot (s' \cdot \x)=(ss')\cdot \x$.

\

\noindent Suppose $X_1$ and $X_2$ are local $1$-parameter subgroups of $G$ with $[X_1]=[X_2]$.  Suppose also that $t\in \dom(X_1)\cap \dom(X_2)$.  Then $X_1(t)=X_2(t)$.  To see this, let $r\in \r^{>0}$ be such that $X_1|(-r,r)=X_2|(-r,r)$ and choose $n>0$ such that $\frac{t}{n}\in (-r,r)$.  Then $X_1(t)=(X_1(\frac{t}{n}))^n=(X_2(\frac{t}{n}))^n=X_2(t)$.  Hence, it makes sense to define, for $\x\in L(G)$, 
$$\dom(\x):=\bigcup_{X\in \x}\dom(X)$$ and for $t\in \dom(\x)$, we will write $\x(t)$ to denote $X(t)$ for any $X\in \x$ with $t\in \dom(X)$.  

\

\noindent \textbf{The Category of Local Groups}

\

\noindent In this subsection, we define the analogs of a few ordinary group theoretic notions in the local group setting.  We will not need most of this material until Section 8.  The presentation given here borrows from ~\cite{MZ} and ~\cite{P}.

\

\begin{df}
Suppose $G=(G,1,\iota, p)$ and $ G'=(G',1',\iota',p')$ are local groups with $\dom(\iota)=\Lambda$, $\dom(p)=\o$, $\dom(\iota')=\Lambda'$ and $\dom(p')=\o'$.  A \textbf{morphism} from $G$ to $G'$ is a continuous function $f:G\rightarrow G'$ such that:
\begin{enumerate}
\item $f(1)=1'$, $f(\Lambda)\subseteq \Lambda'$ and $(f\times f)(\o)\subseteq \o'$,
\item $f(\iota(x))=\iota'(f(x))$ for $x\in \Lambda$, and
\item $f(p(x,y))=p'(f(x),f(y))$ for $(x,y)\in \o$.
\end{enumerate}
\end{df}

\

\noindent Suppose $U_1$ and $U_2$ are open neighborhoods of $1$ in $G$ and suppose that $f_1:G|U_1\rightarrow G'$ and $f_2:G|U_2\rightarrow G'$ are both morphisms.  We say that $f_1$ and $f_2$ are \textbf{equivalent} if there is an open neighborhood $U_3$ of $1$ in $G$ such that $U_3\subseteq U_1\cap U_2$ and $f_1|U_3=f_2|U_3$.  This defines an equivalence relation on the set of morphisms $f:G|U \rightarrow G'$, where $U$ ranges over all open neighborhoods of $1$ in $G$.  We will call the equivalence classes \textbf{local morphisms} from $G$ to $G'$.  To indicate that $f$ is a local morphism from $G$ to $G'$, we write ``$f:G\rightarrow G'$ is a local morphism," even though $f$ is not a function from $G$ to $G'$, but rather an equivalence class of certain partial functions.

\

\noindent Suppose $G,G',G''$ are local groups and $f:G \rightarrow G'$ and $g:G' \rightarrow G''$ are local morphisms.  Choose a representative of $f$ whose image lies in the domain of a representative of $g$.  Then these representatives can be composed and the composition is a morphism.  We denote by $g\circ f$ the equivalence class of this composition.  The equivalence class of the identity function on $G$ is a local morphism and will be denoted by $\id_G$.  We thus have the category \textbf{LocGrp} whose objects are the local groups and whose arrows are the local morphisms with the composition defined above.  A \textbf{local isomorphism} is an isomorphism in the category \textbf{LocGrp}.  Hence, if $f:G \rightarrow G'$ is a local isomorphism, then there is a representative $f:G|U \rightarrow G'|U'$, where $U$ and $U'$ are open neighborhoods of $1$ and $1'$ in $G$ and $G'$ respectively, $f:U\rightarrow U'$ is a homeomorphism, and $f:G|U\rightarrow G'|U'$ and $f^{-1}:G'|U'\rightarrow G|U$ are morphisms.  If there exists a local isomorphism between the local groups $G$ and $G'$, then we say that $G$ and $G'$ are \textbf{locally isomorphic}.  

\

\begin{ex}
One can consider the additive group $\r$ as a local group and $(-r,r)$ for $r\in (0,\infty]$ as a restriction of this local group, so that a local 1-ps $(-r,r)\rightarrow G$ is a morphism of local groups.  In this way, the elements of $L(G)$ are local morphisms from $\r$ to $G$.  Moreover, given any local morphism $f:G \rightarrow G'$, one gets a map of sets $L(f):L(G)\rightarrow L(G')$ given by $L(f)(\x)=f\circ \x$.  It can easily be checked that this yields a functor $L:\textbf{LocGrp}\rightarrow \textbf{Sets}$.
\end{ex}

\

\begin{df} 
A \textbf{sublocal group} of $G$ is a set $H\subseteq G$ containing $1$ for which there exists an open neighborhood $V$ of $1$ in $G$ such that
\begin{enumerate}
\item $H\subseteq V$ and $H$ is closed in $V$;
\item if $x\in H\cap \t$ and $x^{-1}\in V$, then $x^{-1}\in H$;
\item if $(x,y)\in (H\times H)\cap \o$ and $xy\in V$, then $xy\in H$.
\end{enumerate}
With $H$ and $V$ as above, we call $H$ a sublocal group of $G$ with \textbf{associated neighborhood} $V$.  A \textbf{normal sublocal group} of $G$ is a sublocal group $H$ of $G$ with an associated neighborhood $V$ such that $V$ is symmetric and 
\begin{enumerate}
\item[(4)] if $y\in V$ and $x\in H$ are such that $yxy^{-1}$ is defined and $yxy^{-1}\in V$, then $yxy^{-1}\in H$.
\end{enumerate}
For such $H$ and $V$, we also say that $H$ is a normal sublocal group of $G$ with associated \textbf{normalizing} neighborhood $V$.
\end{df}

\

\noindent One checks easily that if $H$ is a sublocal group of $G$ with associated
neighborhood $V$, then
$$(H,1,\iota|H\cap \Lambda \cap \iota^{-1}(V), p|(H\times H) \cap \Omega \cap
p^{-1}(V))$$ is a local group, to be denoted by $H$ for simplicity, and that
then the inclusion $H \hookrightarrow G$ is a morphism.  Sublocal groups $H$ and $H'$ of $G$ are \textbf{equivalent} if there is an open neighborhood $U$ of $1$ in $G$ such that $H\cap U=H'\cap U$.

\

\

\begin{ex}
Suppose $f:G\rightarrow G'$ is a morphism and $\t=G$.  Then $\ker(f):=f^{-1}(\{1\})$ is a normal sublocal group of $G$ with associated normalizing neighborhood $G$.
\end{ex}

\

\noindent The following lemma has a routine verification.

\

\begin{lemma}
Suppose $H$ is a normal sublocal group of $G$ with associated normalizing neighborhood $V$.  Suppose $U\subseteq H$ is open in $H$ and symmetric.  Let $U'\subseteq V$ be a symmetric open neighborhood of $1$ in $G$ such that $U=H\cap U'$. Then $H|U$ is a normal sublocal group of $G$ with associated normalizing neighborhood $U'$. 
\end{lemma}

\

\noindent We now investigate quotients in this category.  For the rest of this subsection, assume that $H$ is a normal sublocal group of $G$ with associated normalizing neighborhood $V$.

\

\begin{lemma}\label{L:localcoset}
Let $W$ be a symmetric open neighborhood of $1$ in $G$ such that $W\subseteq \u_{6}$ and $W^{6}\subseteq V$.  Then
\begin{enumerate}
\item The binary relation $E_H$ on $W$ defined by $$E_H(x,y)\text{ if and only if }x^{-1}y\in H$$ is an equivalence relation on $W$.
\item For $x\in W$, let $xH:=\{xh \ | \ h\in H \text{ and }(x,h)\in \o\}$.  Then for $x,y\in W$, $E_H(x,y)$ holds if and only if $(xH)\cap W = (yH)\cap W$.  In other words, if $E_H(x)$ denotes the equivalence class of $x$, then $E_H(x)=(xH)\cap W$.  We call the equivalence classes \textbf{local cosets} of $H$.
\end{enumerate}
\end{lemma}

\begin{proof}
For (1), first note that the reflexivity of $E_H$ is trivial.  Now suppose $E_H(x, y)$ holds, i.e. that $x^{-1}y\in H$.  Since $(x^{-1}y)^{-1}\in V$, we have that $(x^{-1}y)^{-1}\in H$.  It can easily be verified that $(x^{-1}y)^{-1}=y^{-1}x$, so $E_H(y,x)$ holds.  Finally, suppose $E_H(x, y)$ and $E_H(y, z)$ hold.  It follows from $(x^{-1}y)(y^{-1}z)\in V$ that $(x^{-1}y)(y^{-1}z)\in H$.  However, since $x,y,z\in \u_{4}$, we know that $(x^{-1}y)(y^{-1}z)=x^{-1}z$, so we have transitivity and (1) is proven.

\

\noindent For (2), first suppose that $E_H(x, y)$ holds.  Let $w\in (yH)\cap W$.  Let $h\in H$ be such that $w=yh$.  Then $h=y^{-1}w\in W^2$.  Since $(x^{-1}y)h\in V$, we have $(x^{-1}y)h\in H$.  But $(x^{-1}y)h=x^{-1}w$, so $w\in (xH)\cap W$.  By symmetry, $(xH)\cap W = (yH)\cap W$.  For the converse, suppose $(xH)\cap W = (yH)\cap W$.  Since $y\in (yH)\cap W$, we have $y\in (xH)\cap W$.  Let $h\in H$ be such that $y=xh$.  Then since $xh\in W$, we know that $x^{-1}(xh)$ is defined.  We thus have, by (A), $x^{-1}y=h\in H$ and so $E_H(x, y)$ holds.   
\end{proof}

\

\noindent Let $\pi_{H,W}:W \rightarrow W/E_H$ be the canonical map.  Give $W/E_H$ the quotient topology.  Then $\pi_{H,W}$ becomes an open continuous map.  

\

\noindent Let $\iota_{H,W}:W/E_H\rightarrow W/E_H$ be defined by $\iota_{H,W}(E_H(x))=E_H(x^{-1})$.  Let $\o_{H,W}:=(\pi_{H,W} \times \pi_{H,W})((W\times W)\cap p^{-1}(W))$.  It is easy to check that one can define a map $p_{H,W}:\o_{H,W}\rightarrow W/E_H$ by $p_{H,W}(E_H(x),E_H(y))=E_H(xy)$, where $x$ and $y$ are representatives of their respective local cosets chosen so that $xy\in W$.  The following lemma is easy to verify.      

\

\begin{lemma}
With the notations as above, $$(G/H)_W:=(W/E_H,E_H(1),\iota_{H,W},p_{H,W})$$ is a local group and $\pi_{H,W}:G|W \rightarrow (G/H)_W$ is a morphism.      
\end{lemma}

\

\noindent In the construction of quotients, we have made some choices.  We would expect that, locally, we have the same local group.  The following lemma expresses this.

\

\begin{lemma}
Suppose that $H'$ is also a normal sublocal group of $G$ with associated normalizing neighborhood $V'$ such that $H$ is equivalent to $H'$.  Let $W'$ be a symmetric open neighborhood of $1$ in $G$ used to construct $(G/H')_{W'}$.  Then $(G/H)_W$ and $(G/H')_{W'}$ are locally isomorphic.
\end{lemma}

\begin{proof}
Let $U$ be an open neighborhood of $1$ in $G$ such that $H\cap U=H'\cap U$.  Choose open $U_1$ containing $1$ such that $U_1^2 \subseteq U$.  Let $U_2:=W\cap W'\cap U_1$.  Consider the map $\psi:\pi_{H,W}(U_2)\rightarrow W'/E_{H'}$ given by $\psi(E_H(x))=E_{H'}(x)$ for $x\in U_2$.  One can verify that $\psi:(G/H)_W|\pi_{H,W}(U_2)\rightarrow (G/H')_{W'}$ induces the desired local isomorphism.
\end{proof}

\

\noindent From now on, we write $G/H$ to denote $(G/H)_W$, where $W$ is any open neighborhood of $1$ as in Lemma \ref{L:localcoset}.  We will also write $\pi:G\rightarrow G/H$ to denote the local morphism induced by $\pi_{H,W}:G|W\rightarrow (G/H)_W$.  Since all of the various quotients defined in this way are locally isomorphic, fixing one's attention on one particular local coset space is no loss of generality for our purposes.

\

\

\noindent \textbf{Further Properties of Local Groups}

\

\noindent In this subsection, we first show that we can place two further assumptions on our local groups which lead to no loss of generality in the solution of the Local H5.  Under these new assumptions, we derive several properties of the notion ``$a_1\cdots a_n$ is defined." 

\

\noindent Let $U:=\t \cap \iota^{-1}(\t)$, an open neighborhood of $1$ in $G$.  Note that if $g\in U$, then $\iota(g)\in \t$ and $\iota(\iota(g))=g\in \t$, implying that $\t_U=U$.

\

\noindent \textbf{From this point on, we assume that $\t=G$ for every local group $G$ in this paper.}  By the preceding remarks, every local group has a restriction satisfying this condition and so this is no loss of generality for our purpose.  Note that, with these assumptions, if $(x,y)\in \o$ and $xy=1$, then $x=y^{-1}$ and $y=x^{-1}$.  In particular, $(x^{-1})^{-1}=1$ for all $x\in G$, and $G$ is symmetric. 

\

\

\noindent The following lemma is obvious for topological groups, but requires some care in the local group setting.

\

\begin{lemma}\label{L:homog}(Homogeneity)
\begin{enumerate}
\item For any $g\in G$, there are open neighborhoods $V$ and $W$ of $1$ and $g$ respectively such that $\{g\}\times V\subseteq \o$, $gV\subseteq W$, $\{g^{-1}\}\times W\subseteq \o$, $g^{-1}W\subseteq V$, and the maps $$v\mapsto gv:V\rightarrow W \text{ and }w\mapsto g^{-1}w:W\rightarrow V$$ are each others inverses (and hence homeomorphisms).
\item$G$ is locally compact if and only if there is a compact neighborhood of $1$.
\end{enumerate}
\end{lemma}

\begin{proof}
Clearly (1) implies (2).  For any $g\in G$, define $$\o_g:=\{h\in G \ | \ (g,h)\in \o\}.$$  Then $\o_g$ is an open subset of $G$ and $L_g:\o_g\rightarrow G$, defined by $L_g(h)=gh$, is continuous.  Let $V:=(L_g)^{-1}(\o_{g^{-1}})$.  Then $V$ is open and $1\in V$.  Let $W= L_g(V)\subseteq \o_{g^{-1}}.$  Then $W$ is open since $W=L_{g^{-1}}^{-1}(V)$.  From our choices we can check that $L_g|V$ and $L_{g^{-1}}|W$ are inverses of each other.  
\end{proof}

\

\

\begin{cor}
Let $U= \u_3$.  Then for any $g,h\in U$ such that $(g,h)\in \o_U$, one has $(h^{-1},g^{-1})\in \o_U$ and $(gh)^{-1}=h^{-1}g^{-1}$.
\end{cor}

\begin{proof}
Suppose $g,h\in U$ with $(g,h)\in \o_U$.  Then we have $h^{-1},g^{-1},gh\in U$.  Hence $h^{-1}\cdot g^{-1}\cdot (gh)$ is defined and $$h^{-1}\cdot g^{-1}\cdot (gh)=h^{-1}\cdot (g^{-1}\cdot(gh))=h^{-1}h=1$$ by (A).  But $h^{-1}\cdot g^{-1}\cdot (gh)=(h^{-1}g^{-1})\cdot (gh)$, so $h^{-1}g^{-1}=(gh)^{-1}$.  Also, since $h^{-1}g^{-1}=(gh)^{-1}\in U$, we have $(h^{-1},g^{-1})\in \o_U$.    
\end{proof}

\

\noindent \textbf{From now on, we make the following further assumption on our local group} $G$:  
$$\text{if } (g,h)\in \o \text{, then }(h^{-1},g^{-1})\in \o \text{ and }(gh)^{-1}=h^{-1}g^{-1}.$$ By the preceding corollary, every local group has a restriction satisfying this assumption and so this is no loss of generality for our purpose.

\

\begin{lemma}\label{L:EZFacts}Let $a,a_1,\ldots,a_n\in G$.  Then 

\begin{enumerate}

\item If $a_1\cdots a_n$ is defined and $1\leq i\leq j \leq n$, then $a_i\cdots a_j$ is defined.  In particular, if $a^n$ is defined and $m\leq n$, then $a^m$ is defined.  
\item If $a^m$ is defined and $i,j\in \{0,\ldots,m\}$ are such that $i+j=m$, then $(a^i,a^j)\in \o$ and $a^i \cdot a^j = a^m$.
\item If  $a^n$ is defined and, for all $i,j\in\{1,\ldots,n\}$ with $i+j=n+1$, one has $(a^i,a^j)\in \o$, then $a^{n+1}$ is defined.  More generally, if $a_1\cdots a_n$ is defined, $a_i \cdots a_{n+1}$  is defined for all $i\in\{2,\ldots,n\}$ and $$(a_1\cdots a_i,a_{i+1}\cdots a_{n+1})\in \o \text{ for all }i\in\{1,\ldots ,n\},$$ then $a_1 \cdots a_{n+1}$ is defined. 
\item If $a^n$ is defined, then $(a^{-1})^n$ is defined and $(a^{-1})^n=(a^n)^{-1}$.  (In this case, we denote $(a^{-1})^n$ by $a^{-n}$.)  More generally, if $a_1\cdots a_n$ is defined, then $a_n^{-1}\cdots a_1^{-1}$ is defined and $(a_1\cdots a_n)^{-1}=a_n^{-1}\cdots a_1^{-1}$. 
\item If $k,l\in \z$, $l\not=0$, and $a^{k\cdot l}$ is defined, then $a^k$ is defined, $(a^k)^l$ is defined and $(a^k)^l=a^{k\cdot l}$.   
\end{enumerate}
\end{lemma}

\begin{proof}
The proofs of (1) and (2) are immediate from the definitions.  (3) follows from repeated uses of (A).  We prove the first assertion of (4) by induction on $n$.  The cases $n=1$ and $n=2$ are immediate from our assumptions on local groups.  For the induction step, suppose $a^{n+1}$ is defined and $i,j\in \{1,\ldots,n\}$ with $i+j=n+1$. Then since $(a^j,a^i)\in \o$, by induction we have $((a^{-1})^i,(a^{-1})^j)\in \o$.  Hence, by (3) of the lemma, we know that $(a^{-1})^{n+1}$ is defined.  Also, $$(a^{-1})^{n+1}= (a^{-1})^n\cdot a^{-1}=(a^n)^{-1}\cdot a^{-1}=(a\cdot a^n)^{-1}=(a^{n+1})^{-1}.$$  The second assertion of (4) is proved in the same manner.  

\

\noindent We now prove (5).  It is easy to see that it is enough to check the assertion for $k,l\in \n$, $l\not= 0$.  Since $k\leq k\cdot l$, we have $a^k$ is defined by (1).  We prove the rest by induction on $l$.  This is clear for $l=1$.  Suppose the assertion is true for all $i\leq l$ and suppose $a^{k\cdot (l+1)}$ is defined  To see that $(a^k)^{l+1}$ is defined, we must check that $((a^k)^i,(a^k)^j)\in \o$ for $i,j\in \{1,\ldots,l\}$ such that $i+j=l+1$.  This is the case by (2), since for such $i,j$, we have $k\cdot i + k\cdot j=k\cdot(l+1)$ and by induction $(a^k)^i=a^{k\cdot i}$ and $(a^k)^j=a^{k\cdot j}$.  Now that we know that $(a^k)^{l+1}$ is defined, we must have, by induction, that $$(a^k)^{l+1}=(a^k)^l\cdot a^k=a^{k\cdot l}\cdot a^k=a^{k\cdot l+k}.$$
\end{proof}

\

\noindent By (4) of the previous lemma, we can now speak of $a^k$ being defined for any $k\in \z$, where $a^k$ being defined for $k<0$ means that $a^{-k}$ is defined. 

\

\begin{cor}\label{C:EZCorollary}
Suppose $i,j\in \z$ and $i\cdot j <0$.  If $a^i$ and $a^j$ are defined and $(a^i,a^j)\in \o$, then $a^{i+j}$ is defined and $a^i \cdot a^j=a^{i+j}$.
\end{cor}

\begin{proof}
\noindent The fact that $a^{i+j}$ is defined is clear from the preceding lemma.  Next note that the result is trivial if  $i=-j$.  We only prove the case that $i>0$ and $j<0$ and $i>|j|$, the other cases being similar.  By the preceding lemma, we have $a^i=a^{i+j}\cdot a^{-j}$, whence $a^i \cdot a^j=a^{i+j}$ by (A).
\end{proof}

\

\bigskip \noindent \textbf{The Nonstandard Setting}

\

\noindent We assume familiarity with this setting;
see \cite{D} and \cite{He} for details. Here we just fix notations and terminology.
To each relevant ``basic'' set $S$ corresponds functorially
a set $S^*\supseteq S$,
the {\em nonstandard extension\/} of $S$. In particular, $\n, \r, G$
extend
to $\n^*, \r^*, G^*$,
respectively. Also, any (relevant) relation $R$ and function
$F$ on these basic sets extends functorially to a relation $R^*$ and
function
$F^*$ on the corresponding nonstandard extensions of these basic sets.
For example, the linear ordering $<$ on $\n$ extends to a
linear ordering $<^*$ on $\n^*$, and
the local group operation $p: \o \to G$ of $G$ extends to
an operation $p^*: \o^* \to G^*$. For the sake of
readability
we only use a star in denoting the nonstandard extension of a basic set,
but drop the star when indicating the nonstandard extension
of a relation or function on these basic sets. For example,
when $x,y\in \r^*$ we write $x+y$ and $x<y$ rather than $x+^*y$ and $x<^*
y$.

\

\noindent Given an ambient hausdorff space $S$ and $s\in S$, the {\em monad\/}
of $s$, notation: $\boldsymbol{\mu}(s)$, is by definition the intersection
of all $U^*\subseteq S^*$ with $U$ a neighborhood of $s$ in $S$;
the elements of $\boldsymbol{\mu}(s)$ are the points
of the nonstandard space $S^*$ that are {\em infinitely close\/} to $s$.
The points of $S^*$ that are infinitely close to some $s\in S$ are called
{\em nearstandard}, and $S_{\ns}^*$ is the set of nearstandard
points
of $S^*$:
$$    S_{\ns}^* = \bigcup_{s\in S} \boldsymbol{\mu}(s).$$
Since $S$ is hausdorff, $\muu(s)\cap \muu(s')=\emptyset$ for distinct $s,s'\in S$.  Thus, for $s^*\in S^*_{\ns}$, we can define the \textit{standard part} of $s^*$, notation: $\st(s^*)$, to be the unique $s\in S$ such that $s^*\in \muu(s)$.  For $s_1,s_2 \in S^*_{\ns}$, we put $s_1\sim s_2$ iff $\st(s_1)=\st(s_2)$.

\  

\noindent We let $\boldsymbol{\mu}:=\boldsymbol{\mu}(1)$ be the monad of $1$ in $G^*$; note that $\muu$ is a group with product map $(x,y)\mapsto xy:=p(x,y)$.  

\

\noindent From now on, we let $\nu,\sigma,\tau,\eta,N$ range over $\n^*$ and $i$ and $j$ range over $\z^*$, but \textbf{$m$ and $n$ will always denote elements of $\n$}.  We use the Landau notation as follows:  $i=\lilo(\nu)$ means $|i|<\frac{\nu}{n}$ for all $n>0$ and $i=\bigO(\nu)$ means $|i|<n\nu$ for some $n>0$.  

\

\noindent In our nonstandard arguments, we will assume as much saturation as needed.  The following lemma uses routine nonstandard methods and is left for the reader to verify.

\

\begin{lemma}
$G$ is $\Ns$ if and only if there are no internal subgroups of $\boldsymbol{\mu}$ other than $\{1\}$. 
\end{lemma}

\

\noindent We also consider the internal versions of the notions and results that appear earlier in this section.  In particular, if $(a_1,\ldots,a_\nu)$ is an internal sequence in $G^*$, then it makes sense to say that $(a_1,\ldots,a_\nu)$  (internally) represents $b$, where $b\in G^*$.  We say that $a_1\cdots a_\nu$ is defined if there exists $b\in G^*$ such that $(a_1,\ldots,a_\nu)$ represents $b$.  Likewise, we have the notion ``$a^i$ is defined" for $a\in G^*, i\in \z^*$.

\

\noindent The proof of the following lemma is trivial.

\

\begin{lemma}

\

\begin{enumerate}
\item Suppose $a,b\in G$ and $a'\in \boldsymbol{\mu}(a)$ and $b'\in \boldsymbol{\mu}(b)$.
If $(a,b)\in \o$, then $(a',b')\in \o^*$, $a'\cdot b' \in G^*_{\ns}$, and $\st(a'\cdot b')=a\cdot b$.
\item For any $a\in G^*_{\ns}$ and $b\in \boldsymbol{\mu}$, $(a,b),(b,a)\in \o^*$, $a\cdot b, b\cdot a\in G^*_{\ns}$,  and $\st(a\cdot b)=\st(b\cdot a)=\st(a)$.
\item For any $a,b\in G^*_{\ns}$, if $(a,b^{-1})\in \o^*$ and $ab^{-1}\in \boldsymbol{\mu}$, then $a\sim b$.
\item For any $a\in G$, $a'\in \boldsymbol{\mu}(a)$, and any $n$, if $a^n$ is defined, then $(a')^n$ is defined and $(a')^n\in \boldsymbol{\mu}(a^n)$. 
\end{enumerate}
\end{lemma}

\
  
\noindent A standard consequence of the last item of the preceding lemma is that the partial function $p_n:G\rightharpoonup G$, defined by $p_n(a)=a^n$ if $a^n$ is defined, has an open domain and is continuous.  Note that $\u_n \subseteq \dom(p_n)$.

\

\begin{lemma}\label{L:Overflow}
Suppose $U$ is a neighborhood of $1$ in $G$ and $a\in \boldsymbol{\mu}$.  Then there is a $\nu> \n$ such that $a^\sigma$ is defined and $a^\sigma \in U^*$ for all $\sigma \in \{1,\ldots,\nu\}$.  
\end{lemma}

\begin{proof}
Let $X=\{\sigma \in \n^* \ | \ a^\sigma \text{ is defined and } a^\sigma \in U^* \}$.  Then $X$ is an internal subset of $\n^*$ which contains $\n$ since $\boldsymbol{\mu}$ is a subgroup of $G^*$.  Hence, by overflow, there is a $\nu> \n$ such that $\{0,1,\ldots,\nu\}\subseteq X$.  
\end{proof} 

\

\section{Connectedness and Purity}

\noindent In this section, we describe the different ways powers of infinitesimals can grow and its impact on the existence of connected subgroups of $G$.  We also show how to build local 1-parameter subgroups from ``pure" infinitesimals.  

\

\noindent \textbf{In the rest of this paper, we assume that $G$ is locally compact}.  In particular, $G$ is regular as a topological space.  Also, $a$ and $b$ will denote elements of $G^*$.
 
\

\noindent For $\nu > \n$, we set:
$$G(\nu):=\{a\in \boldsymbol{\mu} \ | \ a^i \text{ is defined and }a^i\in \boldsymbol{\mu} \text{ for all }i=\lilo(\nu)\},$$
$$G^{\lilo}(\nu):=\{a\in \boldsymbol{\mu} \ | \ a^i \text{ is defined and }a^i\in \boldsymbol{\mu} \text{ for all }i=\bigO(\nu)\}.$$
\noindent Note that $G^{\lilo}(\nu)\subseteq G(\nu)$ and both sets are symmetric. 

\

\begin{lemma}\label{L:lilo}
Let $a\in \muu$ and $\nu>\n$.  Then the following are equivalent:
\begin{enumerate}
\item $a^i$ is defined and $a^i\in \muu$ for all $i\in \{1,\ldots,\nu\}$;
\item $a\in G^{\lilo}(\nu)$;
\item there is $\tau \in \{1,\ldots,\nu\}$ such that $\nu =\bigO(\tau)$ and $a^i$ is defined and $a^i\in \muu$ for all $i\in \{1,\ldots,\tau\}$.
\end{enumerate}
\end{lemma}

\begin{proof}
(1) $\Rightarrow$ (2):  Fix $\sigma$ such that $\sigma=\bigO(\nu)$.  We claim that if $a^\sigma$ is defined, then $a^\sigma\in \muu$.  To see this, write $\sigma =n\nu + \eta$ for some $n$ and $\eta < \nu$.  By Lemma \ref{L:EZFacts}, $$a^\sigma =(a^\nu)^n\cdot a^\eta \in \muu \cdot \muu \subseteq \muu.$$  We now prove by internal induction  that $a^i$ is defined for all $i\in \{1,\ldots,\sigma\}$.  The assertion is true for $i=1$ and inductively suppose that $a^i$ is defined and $i+1\leq \sigma$.  By Lemma \ref{L:EZFacts}, in order to show that $a^{i+1}$ is defined, it suffices to show that $(a^k,a^l)\in \o^*$ for all $k,l\in \{1,\ldots,i\}$ with $k+l=i+1$.  However, $(a^k,a^l)\in \muu \times \muu \subseteq \o^*$, finishing the induction.  

\

\noindent Clearly (2)$\Rightarrow$ (1) and (3), so it remains to prove that (3) $\Rightarrow (2)$.  So assume that there is $\tau \in \{1,\ldots,\nu\}$ as in (3).  By the proof of (1)$\Rightarrow (2)$, we see that $a\in G^{\lilo}(\tau)$.  Since $G^{\lilo}(\tau)\subseteq G^{\lilo}(\nu)$, we are finished.   
\end{proof}

\

\begin{df}
We say that $a\in \boldsymbol{\mu}$ is \textbf{degenerate} if, for all $i$, $a^i$ is defined and $a^i\in \boldsymbol{\mu}$.  
\end{df}

\

\noindent It is easy to see that $G$ is $\Ns$ if and only if $\muu$ has no nondegenerate elements other than $1$.

\

\begin{lemma}
Let $a_1,\ldots,a_\nu$ be an internal sequence of elements of $G^*$ with $\nu > 0$ such that for all $i\in \{1,\ldots,\nu\}$ we have $a_i\in \boldsymbol{\mu}$, $a_1\cdots a_i$ is defined and $a_1\cdots a_i\in G^*_{\ns}$.  Then the set $$S:=\{\st(a_1\cdots a_i):1\leq i \leq \nu\}\subseteq G$$ is compact and connected (and contains 1).
\end{lemma}

\

\noindent The proof is as in the global case, but we include it for completeness.
 
\begin{proof}
Suppose for every compact $K\subseteq G$, there is $i\in \{1,\ldots, \nu\}$ such that $a_1\cdots a_i \notin K^*$.  Then, by saturation, there is $i\in \{1,\ldots, \nu\}$ such that $a_1\cdots a_i \notin K^*$ for all compact $K\subseteq G$, contradicting local compactness.  Hence there is compact $K\subseteq G$ with $a_1\cdots a_i\in K^*$ for all $i\in \{1,\ldots, \nu\}$, so $S\subseteq K$.  It is well-known that $S$ must be closed, being the standard part of an internal nearstandard set.  Hence $S$ is  compact.

\

\noindent Suppose $S$ is not connected.  Then we have disjoint open subsets $U$ and $V$ of $G$ such that $S\subseteq U\cup V$ and $S\cap U \not= \emptyset$, $S\cap V \not= \emptyset$.  Assume $1\in U$, so $a_1 \in U^*$.  Choose $i\in \{1,\ldots, \nu\}$ minimal such that $a_1\cdots a_i\in V^*$.  Then $i\geq 2$ and $a_1\cdots a_{i-1}\in U^*$.  Now $a:=\st(a_1\cdots a_{i-1})=\st(a_1\cdots a_i)\in S$.  If $a\in U$, then $a_1\cdots a_{i-1}\in U^*$ and if $a\in V$, then $a_1\cdots a_i \in V^*$; either option yields a contradiction.
\end{proof}

\

\noindent Until further notice, $U$ denotes a compact symmetric neighborhood of $1$ such that $U\subseteq \u_2$.

\

\begin{df}\label{D:ordelement}
Let $a\in G^*$.  If, for all $i$, $a^i$ is defined and $a^i\in U^*$, define ord$_U(a)=\infty$.  Else, define ord$_U(a)=\nu$ if, for all $i$ with $|i|\leq \nu$, $a^i$ is defined, $a^i \in U^*$, and $\nu$ is the largest element of $\n^*$ for which this happens.
\end{df}

\

\noindent It is easy to see that ord$_U(a)=0$ iff $a\notin U^*$.  By Lemma \ref{L:Overflow}, we can see that ord$_U(a)>\n$ if $a\in \boldsymbol{\mu}$.  Note that if $a\in \boldsymbol{\mu}$ and ord$_U(a)=\nu \in \n^*$, then $a^{\nu + 1}$ is defined.  To see this, it is enough to check that $(a^i,a^j)\in \o^*$ for all $i, j \in \{1,\ldots,\nu\}$ such that $i+j=\nu +1$.  However, this is true since $(a^i,a^j)\in U^* \times U^* \subseteq \u_2^*\times \u_2^* \subseteq \o^*$.  Hence, if ord$_U(a)=\nu$, then this is not because $a^{\nu+1}$ is undefined but rather for the more important reason that $a^{\nu+1}$ is defined but not in $U^*$.  We have thus captured the correct adaptation of the global notion to our local setting.

\

\begin{lemma}\label{L:connsubgroup}
Suppose $a\in \boldsymbol{\mu}$ and $a^i \notin \boldsymbol{\mu}$ for some $i=\lilo(\ord_U(a))$.  Then $U$ contains a nontrivial connected subgroup of $G$. 
\end{lemma}

\begin{proof}
By the previous lemma, the set $$G_U(a):=\{\st(a^i) \ | \ |i|=\lilo(\ord_U(a))\}$$ is a union of connected subsets of $U$, each containing $1$, hence is connected.   It is also a subgroup of $G$ by Lemma \ref{L:EZFacts} and Corollary \ref{C:EZCorollary}.
\end{proof}

\

\begin{df}
We say that $a\in \boldsymbol{\mu}$ is \textbf{$U$-pure}\footnote{We have decided to replace the use of ``regular" in \cite{H} with ``pure" as ``regular" already has a different meaning for topological spaces.} if it is nondegenerate and $a\in G(\tau)$, where $\tau:=\ord_U(a)$.  We say that $a\in \boldsymbol{\mu}$ is \textbf{pure} if it is $V$-pure for some compact symmetric neighborhood $V$ of $1$ such that $V\subseteq \u_2$.
\end{df}

\

\noindent The last lemma now reads that if $U$ contains no non-trivial connected subgroup of $G$, then every $a\in \boldsymbol{\mu}$ which is nondegenerate is $U$-pure.

\

\begin{lemma}\label{L:pure}
Let $a\in \boldsymbol{\mu}$.  Then $a$ is pure iff there is $\nu > \n$ such that $a^\nu$ is defined, $a^\nu \notin \boldsymbol{\mu}$ and $a\in G(\nu)$.
\end{lemma}

\begin{proof}
First suppose that $a$ is $V$-pure.  Let $\nu = \ord_V(a)$.  Then certainly $a^\nu$ is defined.  Now since $a^{\nu + 1}$ is defined, $a^\nu \notin \boldsymbol{\mu}$, else $a^{\nu+1}=a^\nu \cdot a \in \boldsymbol{\mu}$, contradicting  $a^{\nu+1}\notin V^*$.  However, $a^i\in \boldsymbol{\mu}$ for $i=\lilo(\nu)$ by definition of $V$-purity.  Conversely, suppose one has $\nu>\n$ such that $a^\nu$ is defined, $a^\nu \notin \boldsymbol{\mu}$ and $a\in G(\nu)$.  Since $a^\nu \notin \boldsymbol{\mu}$, there is a compact symmetric neighborhood $V$ of $1$ such that $V\subseteq \u_2$ and such that $a^\nu \notin V^*$.  Let $\tau:=\ord_V(a)$.  Then $\tau < \nu$, implying that $a\in G(\tau)$ and thus $a$ is $V$-pure.         
\end{proof}

\

\noindent Let $Q$ range over internal subsets of $G^*$ such that $Q\subseteq \boldsymbol{\mu}$, $1\in Q$, and $Q$ is symmetric.  If $\nu$ is such that for every internal sequence $a_1,\ldots,a_\nu$ from $Q$, $a_1\cdots a_\nu$ is defined, then we define $Q^\nu$ to be the internal subset of $G^*$ consisting of all $a_1\cdots a_\nu$ for $a_1,\ldots,a_\nu$ an internal sequence from $Q$.  In this situation, we say that $Q^\nu$ is defined.  

\

\begin{df}
If for all $\nu$, $Q^\nu$ is defined and $Q^\nu \subseteq \boldsymbol{\mu}$, then we say that $Q$ is \textbf{degenerate}.   If $Q^\nu$ is defined and $Q^\nu\subseteq U^*$ for all $\nu$, define ord$_U(Q)=\infty$.  Else, define ord$_U(Q)=\nu$ if $Q^\nu$ is defined, $Q^\nu \subseteq U^*$, and $\nu$ is the largest element of $\n^*$ for which this happens.
\end{df}

\

\noindent Using an overflow argument similar to that in Lemma \ref{L:Overflow}, we have that ord$_U(Q)>\n$.  One can also check that if $\ord_U(Q)\in \n^*$, then $Q^{\ord_U(Q)+1}$ is defined.

\

\begin{lemma}\label{L:singsubgroup}
Suppose $Q^\nu \nsubseteq \boldsymbol{\mu}$ for some $\nu=\lilo(\ord_U(Q))$.  Then $U$ contains a non-trivial connected subgroup of $G$.
\end{lemma}

\begin{proof}
As in the proof of Lemma \ref{L:connsubgroup}, the set $$G_U(Q)=\{\st(a) \ | \ a\in Q^\nu \text{ for some }\nu=\lilo(\ord_U(Q))\}$$ is a union of connected subsets of $U$, each containing $1$, and thus itself is a connected subset of $U$.  To see that $G_U(Q)$ is a group, suppose that $a\in Q^\nu$ and $b\in Q^\eta$ for $\nu,\eta=\lilo(\ord_U(Q))$.  Then $\nu+\eta=\lilo(\ord_U(Q))$ and so the concatenation of the sequences representing $a$ and $b$ represents an element of $G^*$, which must necessarily be $a\cdot b$.  Hence $a\cdot b\in Q^{\nu+\eta}$ and $\st(a)\cdot \st(b)=\st(a\cdot b)\in G_U(Q)$.  That $G_U(Q)$ is closed under inverses is an easy consequence of Lemma \ref{L:EZFacts}.
\end{proof}
  
\

\begin{df}
We say that $Q$ is \textbf{$U$-pure} if $Q$ is nondegenerate and if $Q^\nu\subseteq \boldsymbol{\mu}$ for all $\nu=\lilo(\ord_U(Q))$.  We say that $Q$ is \textbf{pure} if it is $V$-pure for some compact symmetric neighborhood $V$ of $1$ in $G$ such that $V\subseteq \u_2$.
\end{df}  
  
\

\noindent The previous lemma now reads that if $U$ contains no nontrivial connected subgroups of $G$, then every nondegenerate $Q$ as above is $U$-pure.

\

\begin{lemma}
$Q$ is pure iff there is some $\tau > \n$ such that $Q^\tau$ is defined, $Q^\tau \nsubseteq \boldsymbol{\mu}$ and $Q^\nu\subseteq \boldsymbol{\mu}$ for all $\nu=\lilo(\tau)$.  
\end{lemma}

\begin{proof}
First suppose $Q$ is $V$-pure.  Let $\tau=\ord_V(Q)$.  Then $\tau > \n$, $Q^\tau$ is defined, $Q^\tau \nsubseteq \boldsymbol{\mu}$ (otherwise $Q^{\tau +1}\subseteq \boldsymbol{\mu}\subseteq V^*$, contradicting the definition of $\tau$), and $Q^\nu \subseteq \boldsymbol{\mu}$ for $\nu=\lilo(\tau)$ by definition of being $V$-pure.  Conversely, suppose we have $\tau>\n$ such that $Q^\tau$ is defined, $Q^\tau \nsubseteq \boldsymbol{\mu}$, and $Q^\nu\subseteq \boldsymbol{\mu}$ for $\nu=\lilo(\tau)$.  Since $Q^\tau \nsubseteq \muu$, we can find a compact symmetric neighborhood $V$ of $1$ in $G$ such that $V\subseteq \u_2$ and $Q^\tau \nsubseteq V$.  Let $\eta:=\ord_V(Q)$.  Then $\eta < \tau$, implying that $Q^\nu \subseteq \muu$ for all $\nu =\lilo(\eta)$, that is $Q$ is $V$-pure.
\end{proof}

\

\begin{df}
$G$ \textbf{has no small connected subgroups}, henceforth abbreviated $G$ is $\Nsc$, if there is a neighborhood of $1$ in $G$ that contains no connected subgroup of $G$ other than $\{1\}$.  We say that $G$ is \textbf{pure} if every nondegenerate $Q$ as above is pure.  
\end{df}

\

\begin{cor}
If $G$ is $\Nsc$, then $G$ is pure.
\end{cor}

\

\noindent \textbf{Pure infinitesimals and local 1-parameter subgroups}

\

\noindent  Suppose $a\in \boldsymbol{\mu}$ is $U$-pure and $0<\nu=\bigO(\tau)$, where $\tau:=\ord_U(a)$.  Let $$\Sigma_{\nu,a,U}:=\{r\in \r^{>0} \ | \ a^{[r\nu]} \text{ is defined and } a^i\in U^* \text{ if }|i|\leq [r\nu]\}.$$  Since $\nu=\bigO(\tau)$, $\Sigma_{\nu,a,U} \not= \emptyset$.  Let $r_{\nu,a,U}=\sup \Sigma_{\nu,a,U}$, where we let $r_{\nu,a,U}=\infty$ if $\Sigma_{\nu,a,U}=\r^{>0}$.  Suppose $s<r_{\nu,a,U}$.  Then $s\in \Sigma_{\nu,a,U}$ and $a^{[s\nu]}\in U^*$.  Since $a^{[s\nu]+1}$ is defined (see comment after Definition \ref{D:ordelement}; all that was used there was that $U\times U\subseteq \o$), we have that $a^{[s\nu]+1}=a^{[s\nu]}\cdot a\in G^*_{\ns}$.  Hence we have that $a^{[s\nu]}$ is defined and nearstandard for $s\in (-r_{\nu,a,U},0)$ as well since, for such $s$, we have that $[s\nu]=-[(-s)\nu]$ or $-[(-s)\nu]-1$.

\

\begin{lemma}\label{L:regelement}
The map $X:(-r_{\nu,a,U},r_{\nu,a,U})\rightarrow G$ given by $X(s):=\st(a^{[s\nu]})$ is a local 1-ps of $G$.
\end{lemma}

\begin{proof}
Throughout this proof, we let $r:=r_{\nu,a,U}$.  We first show that $X$ satisfies the additivity condition in the definition of a local 1-ps of $G$.  Note that the preceding arguments show that $\im(X)\subseteq \u_2$.  Suppose first that $s_1,s_2,s_1+s_2\in (0,r)$.  Since $[s_1\nu]+[s_2\nu]=[(s_1+s_2)\nu]$ or $[(s_1+s_2)\nu]-1$, we have that $a^{[s_1\nu]+[s_2\nu]}$ is defined and $\st(a^{[(s_1+s_2)\nu]})=\st(a^{[s_1\nu]+[s_2\nu]})$.  Hence $X(s_1+s_2)=X(s_1)\cdot X(s_2)$.  The case that $s_1,s_2,s_1+s_2\in (-r,r)$ with $s_1\cdot s_2 <0$ works similarly, using Corollary \ref{C:EZCorollary}.  Using the fact that $X(-s)=X(s)^{-1}$, we get the additivity property for $s_1,s_2\in (-r,0)$ for which $s_1+s_2\in (-r,0)$.

\

\noindent It remains to show that $X$ is continuous.  By Lemma \ref{L:homog}, it is enough to show that $X$ is continuous at $0$.  Let $V$ be a compact neighborhood of $1$ in $G$.  If $i=\lilo(\nu)$, then $a^i$ is defined and $a^i\in \boldsymbol{\mu} \subseteq V^*$.  By saturation, we have $n$ such that $\frac{1}{n}<r$ and such that if $|\frac{i}{\nu}|<\frac{1}{n}$, then $a^i$ is defined and $a^i \in V^*$.  Hence $X(s)\in V$ for $|s|<\frac{1}{n}$.

\end{proof}

\
 
\noindent We end this section with an extension lemma which will be used later in the paper.  From now on, $I:=[-1,1]\subseteq \r$.

\

\begin{lemma}\label{L:extension}
Suppose $X:I\rightarrow G$ is a continuous function such that for all $r,s\in I$, if $r+s\in I$, then $(X(r),X(s))\in \o$ and $X(r+s)=X(r)\cdot X(s)$.  Further suppose that $X(I)\subseteq \u_4$.  Then there exists $\epsilon \in \r^{>0}$ and a local 1-parameter subgroup $\overline{X}:(-1-\epsilon,1+\epsilon)\rightarrow G$ of $G$ such that $\overline{X}|I=X$.  
\end{lemma} 

\begin{proof}
Fix $\epsilon\in (0,\frac{1}{2})$.  We claim that $\overline{X}:(-1-\epsilon,1+\epsilon)\rightarrow G$ defined by $\overline{X}|I=X$ and, for $0<\delta < \epsilon$, $\overline{X}(1+\delta)=X(1)\cdot X(\delta)$ and $\overline{X}(-1-\delta)=\overline{X}(1+\delta)^{-1}$ satisfies the conclusion of the lemma.  By Lemma \ref{L:homog}, it is enough to show that $(\overline{X}(r),\overline{X}(s))\in \o$ and $\overline{X}(r)\overline{X}(s)=\overline{X}(r+s)$ for all $r,s\in (-1-\epsilon,1+\epsilon)$ for which $r+s\in (-1-\epsilon,1+\epsilon)$.  We first note that $\overline{X}(r),\overline{X}(s)\in \u_4^2$ so that $(\overline{X}(r),\overline{X}(s))\in \o$ is immediate.  The additivity property of $\overline{X}$ splits up into several cases.

\noindent \textbf{Case 1a:}  $r,s,r+s\in [0,1]$.  This case is trivial.

\noindent \textbf{Case 1b:} $r,s\in [0,1+\epsilon)$, $r+s\in (1,1+\epsilon)$.  Then
$$\overline{X}(r+s)=X(1)X(r+s-1)=X(1)X(r-1)X(s)=\overline{X}(r)\overline{X}(s).$$  

\noindent \textbf{Case 2a:}  $r,s,r+s\in [-1,0]$.  This case is trivial.

\noindent \textbf{Case 2b:}  $r,s\in (-1-\epsilon,0]$, $r+s\in (-1-\epsilon,-1)$.  Then
$$\overline{X}(r+s)=\overline{X}(-s-r)^{-1}=(\overline{X}(-s)\overline{X}(-r))^{-1}=\overline{X}(r)\overline{X}(s).$$

\noindent \textbf{Case 3a:}  $r\in [0,1]$, $s\in [-1,0]$.  This case is trivial.

\noindent \textbf{Case 3b:}  $r\in [0,1]$, $s\in (-1-\epsilon,-1)$.  Then
\begin{align}
\overline{X}(r+s)&=\overline{X}(-r-s)^{-1}\notag \\ \notag
			 &=(X(1)X(-r-s-1))^{-1}\\ \notag
			 &=X(r+s+1)X(-1)\\ \notag
			 &=X(r)X(s+1)X(-1)\\ \notag
			 &=X(r)(X(1)X(-s-1))^{-1}\\ \notag
			 &=\overline{X}(r)\overline{X}(s).\notag
\end{align}      

\noindent \textbf{Case 3c:}  $r\in (1,1+\epsilon)$, $s\in [-1,0]$.  Then by Case 3b, we have
$$\overline{X}(r+s)=\overline{X}(-s-r)^{-1}=(\overline{X}(-s)\overline{X}(-r))^{-1}=\overline{X}(r)\overline{X}(s).$$

\noindent \textbf{Case 3d:}  $r\in (1,1+\epsilon)$, $s\in (-1-\epsilon,-1)$.  For this case, we will need the observation that for any $\delta \in (-\epsilon,\epsilon)$, we have $X(1)X(\delta)=X(\delta)X(1)$.  We shall only prove this for $\delta\geq 0$, the proof being similar for $\delta<0$.  The following chain of equalities verifies the observation:

\begin{align}
X(1)X(\delta)&=X(1-\delta)X(\delta)X(\delta)\notag \\ \notag
		    &=X(1-\delta)X(2\delta)\\ \notag
		    &=X(\delta)X(1-2\delta)X(2\delta)\notag \\
		    &=X(\delta)X(1).\notag
\end{align}

Now, we have
\begin{align}
\overline{X}(r+s)&=X(r+s)X(1)X(-1) \notag \\ \notag
	                    &=X(1)X(r+s)X(-1) \\ \notag
	                    &=X(1)X(r-1)X(s+1)X(-1)\\ \notag
	                    &=\overline{X}(r)\overline{X}(s).\notag
\end{align}

\noindent \textbf{Case 4:}  $r\in (-1-\epsilon,0]$, $s\in [0,1+\epsilon)$.  This case is treated exactly like Case 3.
\end{proof}

\

\section{Consequences of $\Ns$} 

\

\noindent In this section we assume $G$ is $\Ns$.  

\

\noindent  \textbf{Special Neighborhoods}

\

\noindent The next lemma is from \cite{K}, and we repeat its proof, with a more detailed
justification of some steps.

\

\begin{lemma}\label{L:specneigh}
There is a neighborhood $V$ of $1$ such that $V \subseteq \u_2$ and for all $x,y\in V$, if $x^2=y^2$, then $x=y$.  
\end{lemma} 

\begin{proof}
Choose a compact symmetric neighborhood $U$ of $1$ in $G$ with $U\subseteq \u_6$ such that $U$ contains no subgroups of $G$ other than $\{1\}$.  Choose a symmetric open neighborhood $W$ of $1$ in $G$ such that $W^5\subseteq U$.  Since $U$ is compact, we can find a symmetric open neighborhood $V$ of $1$ in $G$ such that $V^2\subseteq W$ and $gVg^{-1}\subseteq W$ for all $g\in U$.  We claim that this $V$ fulfills the desired condition.  

\

\noindent Suppose $x,y\in V$ are such that $x^2=y^2$.  Let $a:=x^{-1}y\in V^2\subseteq W \subseteq U$.  Since $V\subseteq \u_5$, we have $$axa=(x^{-1}y)x(x^{-1}y)=x^{-1}y^2=x^{-1}x^2=x.$$

\

\noindent Claim:  For all $n$, $a^n$ is defined, $a^n\in U$ and $a^n=xa^{-n}x^{-1}$.

\

\noindent We prove the claim by induction.  The case $n=1$ follows from the fact that $axa=x$.  For the inductive step, suppose the assertion holds for all $m\in \{1,\ldots,n\}$.  In order to show that $a^{n+1}$ is defined, it suffices to show that $(a^i,a^j)\in \o$ for all $i,j\in \{1,\ldots,n\}$ such that $i+j=n+1$.  However, by the inductive hypothesis, $(a^i,a^j)\in U \times U \subseteq \o$, proving that $a^{n+1}$ is defined.  Since $U\subseteq \u_6$, we have $$a^{n+1}=a^n\cdot a=(xa^{-n}x^{-1})(xa^{-1}x^{-1})=xa^{-n-1}x^{-1}.$$  It remains to show that $a^{n+1}\in U$.  First suppose that $n+1$ is even, say $n+1=2m$.  Then 
$$a^{n+1}=a^m\cdot a^m=xa^{-m}x^{-1}a^m\in xW\subseteq W^2\subseteq U.$$  Now if $n+1$ is odd, then $$a^{n+1}=a^n \cdot a\in W^2\cdot W^2 \subseteq U.$$  

\

\noindent This finishes the proof of the claim.  The claim yields a subgroup $a^\z$ of $G$ with $a^\z\subseteq U$.  Thus $a=1$ and so $x=y$.

\end{proof}

\

\begin{df}
A \textbf{special neighborhood of $G$} is a compact symmetric neighborhood $\u$ of $1$ in $G$ such that $\u\subseteq \u_2$, $\u$ contains no nontrivial subgroup of $G$, and for all $x,y\in \u$, if $x^2=y^2$, then $x=y$.
\end{df}

\

\noindent By the last lemma, we know that $G$ has a special neighborhood.  In the rest of this section, we fix a special neighborhood $\u$ of $G$.  Then every $a\in \boldsymbol{\mu} \setminus \{1\}$ is $\u$-pure.  For $a\in G^*$, we set ord$(a):=\ord_\mathcal{U}(a).$  

\

\begin{lemma}
Let $a\in G^*$.  Then $\ord(a)$ is infinite iff $a\in \boldsymbol{\mu}$.
\end{lemma}

\begin{proof}
We remarked earlier that if $a\in \boldsymbol{\mu}$, then ord$(a)$ is infinite.  Now suppose ord$(a)$ is infinite.  Then, for all $k\in \z$, $a^k$ is defined and $a^k\in \u^*$.  It is easy to prove by induction that $\st(a)^k$ is defined and $\st(a)^k \in \u$ for all $k\in \z$.  This implies that $\st(a)=1$ since $\u$ is a special neighborhood.
\end{proof}

\

\begin{lemma}\label{L:sigparam}
Suppose $\sigma > \n$ and $a\in G(\sigma)$.  Then:
\begin{enumerate}
\item if $a\not= 1$, then $a$ is $\u$-pure and $\sigma=\bigO(\ord(a))$;
\item if $i=\lilo(\sigma)$, then $a^i$ is defined and $a^i \in \boldsymbol{\mu}$;
\item Let $\Sigma_a:=\Sigma_{\sigma,a,\u}$ and $r_a:=r_{\sigma,a,\u}$.  Let $X_a:(-r_a,r_a)\rightarrow G$ be defined by $X(s):=\st(a^{[s\sigma]})$.  Then $X_a$ is a local $1$-ps of $G$.
\end{enumerate}
\end{lemma}

\begin{proof}
It has already been remarked why $a$ is $\u$-pure.  If $\ord(a)=\lilo(\sigma)$, then $a^{\ord(a)}\in \boldsymbol{\mu}$, a contradiction.  Hence, $\sigma =\bigO(\ord(a))$.  By (1), if $i=\lilo(\sigma)$, then $i=\lilo(\ord(a))$, whence $a^i$ is defined and $a^i\in \boldsymbol{\mu}$ since $a$ is $\u$-pure.  This proves (2).  (3) follows immediately from Lemma \ref{L:regelement}.
\end{proof}

\

\begin{cor}
Suppose $G$ is not discrete.  Then $L(G)\not= \{\O\}$.
\end{cor}

\begin{proof}
Take $a\in \boldsymbol{\mu}\setminus \{1\}$.  Set $\sigma:=\ord(a)$.  Since $a\in \boldsymbol{\mu}$, we must have $\sigma > \n$.  Since $a\in G(\sigma)$, we have $[X_a]\in L(G)$, where $X_a$ is as defined in Lemma \ref{L:sigparam}.  We need to show that $[X_a]\not= \O$.  It suffices to show that for every $n$ such that $\frac{1}{n}\in \Sigma_a$, $a^{[\frac{1}{n}\sigma]}\notin \boldsymbol{\mu}$.  Let $t:=[\frac{1}{n}\sigma]$.  Then $t=\frac{1}{n}\sigma-\epsilon$, with $\epsilon \in [0,1)^*$.  Towards a contradiction, suppose that $a^t\in \boldsymbol{\mu}$.  Then since $nt\leq \sigma$, $a^{nt}$ is defined and $a^{nt}=(a^t)^n\in \boldsymbol{\mu}$ by Lemma \ref{L:EZFacts}.  Also $n\epsilon \in \n^*$ and $n\epsilon < n$, whence $a^{n\epsilon}\in\boldsymbol{\mu}$.  But then $a^\sigma=a^{nt+n\epsilon}=a^{nt}\cdot a^{n\epsilon}\in \boldsymbol{\mu}$, a contradiction.
\end{proof}

\

\begin{lemma}\label{L:Gsigma properties}Let $a\in G(\sigma)\setminus \{1\}$.  Then:
\begin{enumerate}
\item $a^{-1}\in G(\sigma)$ and $[X_{a^{-1}}]=(-1)\cdot [X_a]$;
\item $b\in \boldsymbol{\mu} \Rightarrow bab^{-1}\in G(\sigma)$ and $[X_{bab^{-1}}]=[X_a]$;
\item $[X_a]=\O \Leftrightarrow a\in G^{\lilo}(\sigma)$;
\item $L(G)=\{ [X_a] \ | \ a\in G(\sigma)\}$.
\end{enumerate}
\end{lemma}

\begin{proof}
(1) is straightforward.  As to (2), let $b\in \boldsymbol{\mu}$ and let $\tau:=\ord(a)$.  We know that $\sigma = \bigO(\tau)$.  One can show by internal induction on $\eta$ that $(bab^{-1})^\eta$ is defined and equal to $ba^\eta b^{-1}$ for all $\eta \leq \tau$.  Thus, for $i=\lilo(\sigma)$, we have $(bab^{-1})^i \in \boldsymbol{\mu}$ and hence $bab^{-1}\in G(\sigma)$.  Again, for $r\in \Sigma_a$, we have that $(bab^{-1})^{[r\sigma]}$ is defined and equals $ba^{[r\sigma]}b^{-1}$, whence $[X_a]=[X_{bab^{-1}}]$.  

\

\noindent The proof of (3) is immediate from the definitions and Lemma \ref{L:lilo}.  Now suppose $\x\in L(G)$ and $X\in \x$.  Suppose $\dom(X)=(-r,r)$ and consider the nonstandard extension $X:(-r,r)^*\rightarrow G^*$ of $X$.  Let $a=X(\frac{1}{\sigma})\in \boldsymbol{\mu}$.  We claim that $\x=[X_a]$.  To see this, let $s<\min(r,r_a)$ and let $\epsilon$ be an infinitesimal element of $\r^*$ such that $s=\frac{[s\sigma]}{\sigma}+\epsilon$.  Then:

\begin{align}
X(s)&=\st(X(\frac{[s\sigma]}{\sigma}))\notag \\
       &=\st(X(\frac{1}{\sigma})^{[s\sigma]})\notag \\
       &=\st(a^{[s\sigma]})\notag \\
       &=X_a(s) \notag \\ \notag
\end{align}
Hence we have shown that $X$ and $X_a$ agree on $(-r,r)\cap (-r_a,r_a)$.  This proves (4).

\end{proof}

\noindent \textbf{The Local Exponential Map}

\

\noindent Consider the following sets:

$$\k:=\{\x \in L(G) \ | I\subseteq \dom(\x) \text{ and }\x(I)\subseteq \u\}$$ and 

$$K:=\{\x(1) \ | \ \x\in \k\}.$$   

\

\begin{lemma}
For every $\x \in L(G)$, there is $s \in (0,1)$ such that $s \cdot \x \in \k$. 
\end{lemma}

\begin{proof}
Let $X\in \x$ and suppose $X:(-r,r)\rightarrow G$.  If $r\leq 1$, pick any $0<s_1 <r$ and then $s_1 \cdot X:(-\frac{r}{s_1},\frac{r}{s_1})\rightarrow G$ will have $I$ contained in its domain.  If $r>1$, let $s_1=1$.  By continuity of $s_1 \cdot X$, there is $0<s _2 < 1$ such that $(s_2 \cdot (s_1 \cdot X))(I)\subseteq \u$.  So for $s:=s_1\cdot s_2$, $s \cdot \x \in \k$.
\end{proof}

\

\begin{lemma}
The map $\x \rightarrow \x(1):\k \rightarrow K$ is bijective.
\end{lemma}

\begin{proof}
Let $\x_1,\x_2$ be in $\k$ with representatives $X_1,X_2$ containing $I$ in their domains.  Suppose $X_1(1)=X_2(1)$.  Then $(X_1(\frac{1}{2}))^2=(X_2(\frac{1}{2}))^2$, implying that $X_1(\frac{1}{2})=X_2(\frac{1}{2})$ since $\u$ is a special neighborhood.  Inductively, we have $X_1(\frac{1}{2^n})=X_2(\frac{1}{2^n})$ for all $n$, and thus $X_1(\frac{k}{2^n})=X_2(\frac{k}{2^n})$ for all $k\in \z$ such that $|\frac{k}{2^n}|\leq 1$.  By density, we have $[X_1]=[X_2]$.
\end{proof}

\

\noindent From now on, for $\x \in \k$, we let $E(\x)=\x(1)$.  So the last lemma now reads $E:\k \rightarrow K$ is a bijection.  We will have much more to say about this \textit{local exponential map} later and it will be a crucial tool in our solution of the Local H5.

\

\noindent \textbf{A Countable Neighborhood Basis of the Identity}

\

\noindent For $Q$ as earlier, we put $\ord(Q):=\ord_\u(Q)$.  Extend this to the standard setting as follows:  for symmetric $P\subseteq G$ with $1\in P$, let $\ord(P)$ be the largest $n$ such that $P^n$ is defined and $P^n \subseteq \u$ if there is such an $n$ and set $\ord(P):=\infty$ if $P^n$ is defined and $P^n\subseteq \u$ for all $n$.

\

\noindent We set $V_n:=\{x\in G \ |\  \ord (x) \geq n\}$.  Notice that for all $n$, $$p_1^{-1}(\u)\cap \ldots \cap p_n^{-1}(\u)\subseteq V_n \subseteq \u_n.$$

\

\begin{lemma}
$(V_n:n\geq 1)$ is a decreasing sequence of compact symmetric neighborhoods of $1$ in $G$, $\ord(V_n)\rightarrow \infty$ as $n\rightarrow \infty$, and $\{V_n:n\geq 1\}$ is a countable neighborhood basis of $1$ in $G$.
\end{lemma}

\begin{proof}
For $\sigma > \n$, consider the internal set $$V_\sigma:=\{g\in G^* \ | \ord(g)\geq \sigma\}.$$  Note that $V_\sigma \subseteq \boldsymbol{\mu}$, so given any neighborhood $U$ of $1$ in $G$, we have $x_1\cdots x_m$ is defined and in $U^*$ for all $m$ and $x_1,\ldots,x_m\in V_\sigma$.  It follows that for any neighborhood  $U$ of $1$ in $G$ and any $m$ we have $(V_n)^m$ is defined and contained in $U$ for all sufficiently large $n$.
\end{proof}

\

\section{Local Gleason-Yamabe Lemmas}

\noindent The trickiest and most technical part of the original solution of Hilbert's fifth problem concerned several ingenious lemmas proved by Gleason and Yamabe.  These lemmas really drive much of the proof and obtaining local versions of them is a crucial part of our current project.  Fortunately, it is not too difficult to obtain local versions of these lemmas, although carrying out some of the details can become laborious.

\

\noindent Throughout this section, we work with the following assumptions.  We suppose $\u \subseteq \w$ are compact symmetric neighborhoods of $1$ such that $\w \subseteq \u_M$ and $\u^P \subseteq \w$ for fixed positive integers $P\leq M$ large enough for all of the arguments below to make sense.  (A careful analysis of the arguments below could give precise values for $M$ and $P$, but this endeavor became quite tedious.)  We let $Q\subseteq \u$ be symmetric such that $1\in Q$, $N:=\ord_\u(Q)\not=\infty$, and $Q^{N+1}$ is defined.  We first define the map $\Delta=\Delta_Q:G\rightarrow [0,1]$ by 
\begin{itemize}
\item $\Delta(1)=0$;
\item $\Delta(x)=\frac{i}{N+1}$ if $x\in Q^i \setminus Q^{i-1}$, $1\leq i \leq N$;
\item $\Delta(x)=1$ if $x\notin Q^N$.
\end{itemize}

\noindent Then for all $x\in G$, 

\begin{itemize}
\item $\Delta(x)=1$ if $x\notin \u$;
\item $|\Delta(ax)-\Delta(x)|\leq \frac{1}{N}$ for all $a\in Q$ such that $(a,x)\in \o$.
\end{itemize}

\

\noindent To smooth out $\Delta$, fix a continuous function $\tau:G\rightarrow [0,1]$ such that 
$$\tau(1)=1, \qquad \tau(x)=0 \text{ for all }x\in G\setminus \u.$$

\

\noindent It is important later that $\tau$ depends only on $\u$, not on $Q$. 

\

\noindent Next define $\theta=\theta_Q:G\rightarrow [0,1]$ by setting

\[
	\theta(x)=
	\begin{cases}
	\sup_{y\in \u}(1-\Delta(y))\tau(y^{-1}x)	&\text{if $x\in \w$}\\
	0 	&\text{if $x\in G\setminus \w$}
	\end{cases}
\]

\

\begin{lemma} The following properties hold for the above defined functions.
\begin{enumerate}
\item $\theta$ is continuous and $\theta(x)=0$ if $x\notin \u^2$;
\item $0\leq \tau \leq \theta \leq 1$;
\item $|\theta(ax)-\theta(x)|\leq \frac{1}{N}$ for $a\in Q$ if $(a,x)\in \o$.
\end{enumerate}
\end{lemma}

\begin{proof}
It is clear that $\theta(x)=0$ if $x\notin \u^2$.  To see that $\theta$ is continuous, let $x\in G$ and $a\in \boldsymbol{\mu}$.  We must show that $\theta(ax)-\theta(x)$ is infinitesimal.  First note that if $x\notin \w$, then $ax\notin \w^*$; in this case, $\theta(ax)=\theta(x)=0$.  So assume $x\in \w$.  If $ax\notin \w$, then $x\notin \interior(\w)$.  Since $\u^2 \subseteq \interior(\w)$, we conclude that $x\notin \u^2$ and so $\theta(x)=\theta(ax)=0$.  Now suppose $ax\in \w^*$, so we can use the first clause in the definition of $\theta$ for $\theta(x)$ and $\theta(ax)$.  In this case, $\theta(ax)-\theta(x)$ is infinitesimal by the continuity of $\tau$.  $(2)$ is also clear from the definition.  

\

\noindent We now prove (3).  Let $a\in Q$ and suppose $(a,x)\in \o$.  If $x\notin \w$, then $ax\notin \u^2$, else $x\in \u^3 \subseteq \w$, a contradiction.  Hence $\theta(ax)=\theta(x)=0$ in this case.  Now if $x\in \w \setminus \u^3$, then $ax\notin \u^2$, and once again $\theta(ax)=\theta(x)=0$.  Hence we can suppose $x\in \u^3$ so that $ax\in \u^4\subseteq \w$ and so we can use the first clause in the definition of $\theta$ for $\theta(x)$ and $\theta(ax)$.  Let $y\in \u$.  Then $(a^{-1},y)\in \o$ and 
$$|(1-\Delta(a^{-1}y))-(1-\Delta(y))|\leq \frac{1}{N}.$$  Since $y^{-1}ax$ is defined, we have $y^{-1}ax=(a^{-1}y)^{-1}x$, so 
$$|(1-\Delta(y))\tau(y^{-1}ax)-(1-\Delta(a^{-1}y)\tau((a^{-1}y)^{-1}x))|\leq \frac{1}{N}.$$  (3) now follows from noting that the sets $\{(1-\Delta(y))\tau(y^{-1}x) \ | \ y \in \u\}$ and $\{(1-\Delta(a^{-1}y))\tau((a^{-1}y)^{-1}x) \ | \ y \in \u, a^{-1}y\in \u\}$ have the same supremum.
\end{proof}

\

\noindent Let $C:=\{f:G\rightarrow \r \ | \ f \text{ is continuous and }\supp(f)\subseteq \w^2\}.$  Then $C$ is a real vector space and we equip it with the norm given by 
$$\|f\|:=\sup \{|f(x)| \ | \ x\in G\}.$$

\noindent If $f\in C$ is such that $\supp(f)\subseteq \w$ and $a\in \w$, then we can define the new function $a\cdot f$ by the rule

\[
	(a\cdot f)(x)=
	\begin{cases}
	f(a^{-1}x)		&\text{if $x\in \w^2$}\\
	0			&\text{otherwise}.
	\end{cases}
\]

\noindent It is easy to see that $a\cdot f\in C$ and that if $f,g\in C$ both have supports contained in $\w$, then $a\cdot (f+g)=a\cdot f+a\cdot g$.  We will often drop the $\cdot$ and just write $af$ instead of $a\cdot f$.  This should not be confused with the function $rf$ given by $(rf)(x)=rf(x)$ for $r\in \r$, which is defined for any $f:G\rightarrow \r$.  

\

\begin{lemma}\label{L:triangle}
Suppose $a,b\in \w$ and $f\in C$ is such that $\supp(f)\subseteq \w$ and $\supp(bf)\subseteq \w$.  Then:
\begin{enumerate}
\item[(i)] $\|af\|=\|f\|$;
\item[(ii)] $a(bf)=(ab)f$ and $$\|(ab)f -f\|\leq \|af -f\|+\|bf-f\|.$$
\end{enumerate}
\end{lemma}

\begin{proof}
To see (i), note that 
\begin{align}
\|af\|&=\sup \{|f(a^{-1}x)| \ | \ x\in \w^2\}\notag \\
       &=\sup \{|f(a^{-1}x)| \ | \ x\in \w^2, \ a^{-1}x\in \w\}\notag \\
       &=\sup \{|f(y)| \ | \ y\in \w\}\notag \\
       &=\|f\| \notag \\ \notag
\end{align}

\noindent We now prove (ii).  Suppose $x\notin \w^2$.  In this case, $((ab)f)(x)=0$ and since $a^{-1}x\notin \w$ and $\supp(bf)\subseteq \w$, we also have $(a(bf))(x)=0$.  Now suppose $x\in \w^2$.  Then $((ab)f)(x)=f((ab)^{-1}x)=f((b^{-1}a^{-1})x)$ and $(a(bf))(x)=(bf)(a^{-1}x)$.  If $a^{-1}x\notin \w^2$, then $b^{-1}a^{-1}x\notin \w$;  in this case, we have $((ab)f)(x)=0$ and $(a(bf)(x)=0$.  If $a^{-1}x\in \w^2$, then $(a(bf))(x)=(bf)(a^{-1}x)=f(b^{-1}(a^{-1}x))=((ab)f)(x)$.  We now finish the proof of (ii) as in the global case, that is

\begin{align}
\|(ab)f-f\|&=\|a(bf)-af+af-f\| \notag \\
              &\leq \|a(bf)-af\| + \|af-f\| \notag \\
              &\leq \|a(bf-f)\|+\|af-f\| \notag \\
              &=\|bf-f\|+\|af-f\| \notag \\ \notag
\end{align}
\end{proof}

\

\noindent Since $\supp(\theta)\subseteq \u^2\subset \w$, we can consider the function $a\theta\in C$ for any $a\in \w$.  We will often need the following equicontinuity result.
\begin{enumerate}
\item[(4)]For each $\epsilon \in \r^{>0}$, there is a symmetric neighborhood $V_{\epsilon}$ of $1$ in $G$, independent of $Q$, such that $V_{\epsilon}\subseteq \u$ and $\|a\theta -\theta \|\leq \epsilon$ for all $a\in V_{\epsilon}$.
\end{enumerate}

\noindent To see this, let $\epsilon \in \r^{>0}$ and note by the uniform continuity of $\tau$, we have a neighborhood $U$ of $1$ in $G$ such that $|\tau(g)-\tau(h)|<\epsilon$ for all $g,h\in \u$ such that $gh^{-1}\in U$.  Now take a symmetric neighborhood $V_{\epsilon}$ of $1$ in $G$ such that $V_{\epsilon}\subseteq \u$ and $y^{-1}ay\in U$ for all $(a,y)\in \u \times V_{\epsilon}$.  The claim is that this choice for $V_{\epsilon}$ works.  Fix $a\in V_\epsilon$.  We must compute $|(a\theta)(x)-\theta(x)|$ for various $x$.  Let us first take care of some of the trivial computations.  If $x\notin \w^2$, then $(a\theta)(x)=\theta(x)=0$.  Next suppose that $x\in \w^2 \setminus \w$, so $\theta(x)=0$.  If $a^{-1}x\notin \w$, then $(a\theta)(x)=0$.  If $a^{-1}x\in \w$, then $a^{-1}x\notin \u^2$, else $x\in \u^3\subseteq \w$, a contradiction.  So in this case also, $(a\theta)(x)=0$.  Now if $x\in \w \setminus \u^3$, then $a^{-1}x\notin \u^2$, so once again $(a\theta)(x)=\theta(x)=0$.  Thus the only real case of interest is when $x\in \u^3$ and hence $a^{-1}x\in \u^4\subseteq \w$.  In this case, we have $|\tau(y^{-1}a^{-1}x)-\tau(y^{-1}x)|<\epsilon$ for any $y\in \u$.  We now have

\begin{align}
|(a\theta)(x)-\theta(x)|&=\big|\sup_{y\in \u} (1-\Delta(y))\tau(y^{-1}a^{-1}x)-\sup_{y\in \u}(1-\Delta(y))\tau(y^{-1}x)\big| \notag \\
				&\leq \sup_{y\in \u}(1-\Delta(y))|\tau(y^{-1}a^{-1}x)-\tau(y^{-1}x)| \notag \\
				&\leq\epsilon. \notag \\ \notag
\end{align}

\noindent This finishes the proof of (4).

\

\noindent In ~\cite{MZ}, Montgomery and Zippin construct a linear functional $$f\mapsto \int f:C \rightarrow \r$$ with the following properties:

\begin{enumerate}
\item[(i)] For any $f\in C$, one has $|\int f|\leq \int |f|$;
\item[(ii)] (Left invariance) Suppose $f\in C$ satisfies $\supp(f)\subseteq \w$.  Let $a\in \w$.  Then $\int af=\int f$.
\end{enumerate}

\noindent The above defined functional will be referred to as the \textbf{local Haar integral} on $G$.  As is the case for topological groups, the local Haar integral is unique up to multiplication by a positive real constant.   By the Riesz Representation Theorem, for each local Haar integral as above, we get a Borel measure $\mu$ on a certain $\sigma$-algebra of subsets of $\w^2$.  We will refer to the measure corresponding to a local Haar integral as a \textbf{local Haar measure} on $G$.  If a given integral induces the measure $\mu$, then we will denote the integral of a function $f$ by $\int f d \mu$, or if we wish to specify the variable, by $\int f(x)d\mu(x)$.  By the left invariance of the integral, the measure $\mu$ is left-invariant in the sense that for any measurable subset $V\subseteq \w$ and any $a\in \w$, we have $\mu(aV)=\mu(V)$.

\

\noindent With these remarks behind us, we can now proceed with our discussion.  Fix a continuous function $\tau_1:G\rightarrow [0,1]$ such that 
$$\tau_1(x)=1 \text{ for all }x\in \u^2, \qquad \tau_1(x)=0 \text{ for all }x\in G\setminus \u^3.$$

\noindent It is important later that $\tau_1$ depends only on $\u$, not on $Q$.

\

\noindent Note that $0\leq \theta \leq \tau_1$.  Take the unique local Haar measure $\mu$ such that   $\int \tau_1(x)d\mu(x)=1$.  Then in particular we have 

\

\begin{enumerate}
\item[(5)] $0\leq \int \theta(x)d\mu(x) \leq 1.$
\end{enumerate}

\

\noindent By (4) above, we have an open neighborhood $V\subseteq \u$ of $1$ in $G$, independent of $Q$, such that $\theta^2(x)\geq \frac{1}{2}$ on $V$.  We now introduce the function $$\phi=\phi_Q:G\rightarrow \r$$ defined by 

\[
	\phi(x)=
	\begin{cases}
	\int \theta(xu)\theta(u)d\mu(u)	&\text{if $x\in \w$}\\
	0 	&\text{if $x\in G\setminus \w$}
	\end{cases}
\]

\noindent It should be clear that $\phi$ is continuous.  The following are other useful properties of $\phi$:

\begin{enumerate}
\item[(6)]$\supp(\phi)\subseteq \u^4\subseteq \w$;
\item[(7)] $\phi(1)\geq \frac{\mu(V)}{2}$;
\item[(8)] if $a\in Q$, then $\|a\phi - \phi\| \leq \frac{1}{N}$;
\item[(9)] $\|a\phi - \phi\|\leq \|a\theta -\theta\|$ for all $a\in \w$.
\end{enumerate}

\noindent In verifying properties (8) and (9), one needs to perform the case distinctions that we have done earlier.  As before, though, there is only one case where the first clause of the definition applies for both $a\phi$ and $\phi$ and in this case the results easily follow from the properties of the local Haar integral and the earlier derived properties (1)-(5).  We use that $\u^6\subseteq \w$ in showing that there is only one genuine calculation to perform.

\

\noindent We should also point out that, by (6), $\supp(a\phi-\phi)\subseteq \u^5\subseteq \w$ for all $a\in \w$, and so for all $b\in\w$, we can consider the function $b(a\phi-\phi)$.  We are finally in the position to prove the local versions of the Gleason-Yamabe Lemmas.

\begin{lemma}\label{L:gleason1}
Let $c,\epsilon \in \r^{>0}$.  Then there is a neighborhood $U=U_{c,\epsilon}\subseteq \u$ of $1$ in $G$, independent of $Q$, such that for all $a\in Q$, $b\in U$, and $m\leq cN$,
$$\|b\cdot m(a\phi-\phi)-m(a\phi-\phi)\|\leq \epsilon, \qquad \|m(a\phi-\phi)\|\leq c.$$
\end{lemma}

\begin{proof}
Let $a\in Q$, $b\in \u$.  Note that $\supp(b\cdot m(a\phi-\phi)-m(a\phi-\phi))\subseteq \u^6$.  So let $x\in \u^6$ and set $y:=b^{-1}x$.  Then
\begin{align*}  (a\phi - \phi)(x)\ &=\ \int [\theta(a^{-1}xu)-\theta(xu)]\theta(u)\ d\mu(u)\\ 
b(a\phi-\phi)(x)=(a\phi - \phi)(y)\ &=\ \int [\theta(a^{-1}yu)-\theta(yu)]\theta(u)\ d\mu(u).
\end{align*}

\noindent Assuming $P$ is large enough so that $y^{-1}x\in \w$, we can use left-invariance of the integral to replace $u$ by $x^{-1}yu$ in the function of $u$ integrated in the first identity.  So,
$$(a\phi-\phi)(x)=\int [\theta(a^{-1}yu)-\theta(yu)]\theta(x^{-1}yu)d\mu(u).$$

\noindent Taking differences gives 
$$[b\cdot(a\phi - \phi)-(a\phi-\phi)](x)=\int [(a\theta-\theta)(yu)][(\theta-y^{-1}x\theta)(u)]d\mu(u).$$

\noindent By (4), we can take the neighborhood $U_{c,\epsilon}\subseteq \u$ of $1$ in $G$ so small that for all $b\in U_{c,\epsilon}$ and $x\in\u^6$ we have $y^{-1}x\in \u$ and 
$$\|\theta-y^{-1}x\theta\|<\frac{\epsilon}{c\mu(\u^3)}.$$

\noindent It is straightforward to see that this choice of $U_{c,\epsilon}$ works.
\end{proof}

\

\begin{lemma}\label{L:gleason2}
With $c,\epsilon \in \r^{>0}$, let $U=U_{c,\epsilon}$ be as in the previous lemma and let $a\in Q$ and $m,n$ be such that $m\leq cN$, $n>0$, $a^n$ is defined, and $a^i\in U$ for $i\in\{0,\ldots,n\}$.  Then
$$\|(\frac{m}{n})(a^n\phi - \phi)-m(a\phi-\phi)\|\leq \epsilon.$$
\end{lemma}

\begin{proof}
It is routine to check that $m(a^n\phi - \phi)=\sum_{i=0}^{n-1} a^im(a\phi-\phi)$.  Hence we have 
$$m(a^n\phi -\phi)-nm(a\phi-\phi)=\sum_{i=0}^{n-1}(a^im(a\phi-\phi)-m(a\phi - \phi)).$$ 

\noindent By the previous lemma, we have for $i=0,\ldots,n-1$,
$$\|a^im(a\phi-\phi)-m(a\phi-\phi)\|\leq \epsilon,$$ which gives the desired result by summing and dividing by $n$.
\end{proof}

\

\noindent Now suppose that $Q$ is a symmetric internal subset of $\boldsymbol{\mu}$ with $1\in Q$ such that $N:=\ord_\u(Q)\in \n^*$.  Note that $N>\n$.  Then everything we have done so far in this section transfers to the nonstandard setting and yields internally continuous functions $\theta:G^*\rightarrow [0,1]^*$ and $\phi:G^*\rightarrow \r^*$ satisfying the internal versions of (1)-(9) and Lemmas \ref{L:gleason1} and \ref{L:gleason2}.  We now have the following lemma.

\begin{lemma}\label{L:gleason3}
Suppose $a\in Q$, $\nu=\bigO(N)$, $a^\nu$ is defined and $a^\sigma \in \boldsymbol{\mu}$ for all $\sigma \leq \nu$.  Then $\|\nu(a\phi-\phi)\|$ is infinitesimal.
\end{lemma}

\begin{proof}
Let $\epsilon \in \r^{>0}$.  Choose $c\in \r^{>0}$ such that $\nu \leq cN$.  Then by Lemma \ref{L:gleason2}, with $m=n=\nu$, we have $$\|(a^\nu\phi -\phi)-\nu(a\phi-\phi)\|\leq \epsilon.$$

\noindent Also $\|a^\nu \phi-\phi\|\leq \|a^\nu \theta -\theta\|<\epsilon$ by (9) and (4), so $\|\nu(a\phi-\phi)\|\leq 2\epsilon.$
\end{proof}

\

\section{Consequences of the Gleason-Yamabe Lemmas} 

\noindent  In this section, we bear the fruits of our labors from the previous section.  We derive some consequences of the Gleason-Yamabe lemmas which allow us to put a group structure on $L(G)$.  From this group structure on $L(G)$, we obtain the important corollary that in an $\Ns$ local group, the image of the local exponential map is a neighborhood of $1$, which later will allow us to conclude that $\Ns$ local groups are locally euclidean.

\

\noindent \textbf{Group structure on $L(G)$}

\

\begin{lemma}\label{L:prodmonad}
Suppose $\nu\in \n^*\setminus \n$ and $a_1,\ldots,a_\nu$ is a hyperfinite sequence such that $a_i\in G^{\lilo}(\nu)$ for all $i\in \{1,\ldots,\nu\}$.  Let $Q:=\{1,a_1,\ldots,a_\nu,a_1^{-1},\ldots,a_\nu^{-1}\}$.  Then $Q^\nu$ is defined and $Q^\nu \subseteq \boldsymbol{\mu}$.
\end{lemma}

\begin{proof}
We first show that if $Q^\nu$ is defined, then $Q^\nu \subseteq \boldsymbol{\mu}$.  Suppose, towards a contradiction, that $Q^\nu \nsubseteq \muu$.  Take a compact symmetric neighborhood $U$ of $1$ in $G$ such that $Q^{\nu+1}\nsubseteq U^*$, so $\ord_U(Q)\leq \nu$.  Fix $\w$ as in the setting of the Gleason-Yamabe lemmas and choose $U$ small enough so that $U\subseteq \w$ and $U^P\subseteq \w$.  By decreasing $\nu$ if necessary, and $Q$ accordingly, we arrange that $\ord_U(Q)=\nu$.

Consider first the special case that $Q^i\subseteq \muu$ for all $i=\lilo(\nu)$.  Take $b\in Q^\nu$ such that $\st(b)\not=1$.  Choose $\u$ as in the setting of the Gleason-Yamabe lemmas so that $\u\subseteq U$ and $\st(b)\notin \u^4$.  Let $\eta:=\ord_{\u}(Q)$ and note that $\nu = \bigO(\eta)$, else we would have that $Q^{\eta}\subseteq \boldsymbol{\mu}$, a contradiction.  By the Transfer Principle, the previous section yields the internally continuous function $\phi=\phi_Q\colon G^*\rightarrow \r^*$ satisfying the internal versions of all of the properties and lemmas from that section.  In particular, $\phi(b^{-1})=0$ and $\phi(1)$ is not infinitesimal.  Hence $\|b\phi - \phi\|$ is not infinitesimal.  Take a hyperfinite sequence $b_1,\ldots,b_\nu$ from $Q$ such that $b=b_1\cdots b_\nu$.  Then by the Transfer Principle and Lemma \ref{L:triangle}, we have $$\|b\phi -\phi\|\leq \sum_{i=1}^\nu \|b_i\phi -\phi\|$$ 
\noindent and the right hand side of the inequality is infinitesimal by Lemma \ref{L:gleason3}, yielding a contradiction.

Now assume that $Q^i\nsubseteq \muu$ for some $i=\lilo(\nu)$.  Then the compact subgroup $G_U(Q)$ of $G$ as defined in Lemma \ref{L:singsubgroup} is nontrivial (and contained in $U$).  By a standard consequence of the Peter-Weyl Theorem (see \cite{MZ}, page 100), we can find a proper closed normal subgroup $K$ of $G_U(Q)$ and a compact symmetric neighborhood $V\subseteq U$ of $1$ in $G$ such that $K\subseteq \interior(V)$ and every subgroup of $G_U(Q)$ contained in $V$ is contained in $K$.  Put $\mu:=\ord_V(Q)$, so $\mu\in \n^*\setminus \n$ and $\mu\leq \nu$.  Since the compact subgroup $G_V(Q)$ is contained in $V$, we must have $G_V(Q)\subseteq K$.  Take $b\in Q^\mu$ such that $\st(b)\notin \interior(V)$.  Then $\st(b)\notin K$, so we can take a compact symmetric neighborhood $\u$ of $1$ as in the setting of the Gleason-Yamabe lemmas such that $K\subseteq \interior(\u)$, $\u\subseteq V$, and $\st(b)\notin \u^4$.  Put $\eta:=\ord_{\u}(Q)$.  If $\eta=\lilo(\mu)$, then $\st(Q^\eta)\subseteq G_V(Q)\subseteq K$, contradicting the fact that $\st(Q^\eta)\nsubseteq \interior(\u)$.  Hence $\mu=\bigO(\eta)$.  The rest of the proof now proceeds as in the special case considered earlier, with $\nu$ replaced by $\mu$, and $b_1,\ldots,b_\nu$ by an internal sequence $b_1,\ldots,b_\mu$ in $Q$ such that $b=b_1\cdots b_\mu$.

We now prove that $Q^\nu$ is defined.  To do this, we first fix an internal set $E\subseteq \muu$.

\noindent \textbf{Claim:}  Let $c_1,\ldots,c_\eta$ be an internal sequence such that $c_i^j$ is defined and $c_i^j\in E$ for all $i,j\in \{1,\ldots,\eta\}$.  Let  $R=\{1,c_1,\ldots,c_\eta,c_1^{-1},\ldots,c_{\eta}^{-1}\}$.  Then $R^{\eta}$ is defined.

We prove this claim by internal induction on $\eta$.  The claim is clearly true for $\eta=1$ and we now suppose inductively that it holds for a given $\eta$.  Suppose $c_1,\ldots,c_{\eta+1}$ is an internal sequence such that $c_i^j$ is defined and $c_i^j \in E$ for all $i,j\in \{1,\ldots,\eta+1\}$.  Let $R=\{1,c_1,\ldots,c_{\eta+1},c_1^{-1},\ldots,c_{\eta+1}^{-1}\}$.  Fix $d_1,\ldots,d_{\eta+1}\in R$.  By the induction hypothesis and the first part of the proof, $d_1\cdots d_\eta$ is defined and infinitesimal; similarly, $d_1\cdots d_i$ and $d_i\cdots d_{\eta+1}$ are defined and infinitesimal for all $i\in \{2,\ldots,\eta+1\}$.  By Lemma \ref{L:EZFacts}, $d_1\cdots d_{\eta+1}$ is defined.  Since $d_1,\ldots,d_{\eta+1}$ from $R$ were arbitrary, we have that $R^{\eta+1}$ is defined.  This finishes the proof of the claim.

Now given an internal sequence $a_1,\ldots,a_\nu$ as in the statement of the lemma, the claim implies that $Q^\nu$ is defined by taking $E$ to be the set of all $a_i^j$ for $i,j\in \{1,\ldots,\nu\}$.
\end{proof}

\

\begin{nrmk}
In the global version of this lemma appearing in \cite{H}, it was required that $G$ was pure.  Van den Dries was able to remove this assumption and in the proof of the previous lemma we have adapted his method to the local setting. 
\end{nrmk}

\

\begin{lemma}\label{L:Gsigmagroup}
Suppose $U$ is a compact symmetric neighborhood of $1$ in $G$ with $U\subseteq \u_2$. Let $\nu > \n$ be such that for all $i\in\{1,\ldots,\nu\}$, $a^i$ and $b^i$ are defined and $a^i \in U^*$, $b^i \in \boldsymbol{\mu}$.  Then for all $i\in\{1,\ldots,\nu\}$, we have that $(ab)^i$ is defined and $(ab)^i\sim a^i$. 
\end{lemma}

\begin{proof}
Suppose $i\in\{1,\ldots,\nu\}$.  Then $(a^i,b)\in \o^*$ and $a^ib\sim a^i$.  Since $(\st(a^i),\st(a^{-i}))\in \o$, we have $(a^ib,a^{-i})\in \o^*$.  Similarly, $(a^i,ba^{-i})\in \o^*$.  Thus we can define the element $b_i:=a^iba^{-i}$.  For similar reasons, for any $i,j\in\{1,\ldots,\nu\}$, we can define the element $a^ib^ja^{-i}$.  Note that $b_i\in \boldsymbol{\mu}$ for all $i\in\{1,\ldots,\nu\}$.  

\

\noindent Claim 1:  For all $i,j\in\{1,\ldots,\nu\}$, $b_i^j$ is defined and $b_i^j=a^ib^ja^{-i}$.

\

\noindent We prove Claim 1 by internal induction on $j$.  The case $j=1$ is obvious.  Suppose the assertion is true for all $j'\in \{1,\ldots,j\}$ and further suppose $j+1\leq \nu$.  We first show that $b_i^{j+1}$ is defined.  Let $k,l\in\{1,\ldots,j\}$ be such that $k+l=j+1$.  By the induction hypothesis, we know that $b_i^k$ is defined and $b_i^k=a^ib^ka^{-i}$.  Since $b^k\in \boldsymbol{\mu}$, we have $b_i^k \in \boldsymbol{\mu}$.  Similarly, we have $b_i^l$ is defined and $b_i^l\in \boldsymbol{\mu}$; thus $(b_i^k,b_i^l)\in \o^*$.  So by Lemma \ref{L:EZFacts}, we have that $b_i^{j+1}$ is defined.  But now, using that $b_i^j \in \boldsymbol{\mu}$, we have 

\begin{align}
b_i^{j+1}&=b_i^j \cdot b_i \notag \\
	      &=(a^ib^ja^{-i})(a^iba^{-i}) \notag \\
	      &=((a^ib^ja^{-i})\cdot (a^i))\cdot (ba^{-i}) \notag \\
	      &=(a^ib^j)\cdot (ba^{-i}) \notag \\
	      &=((a^ib^j)\cdot b)a^{-i}\notag \\
	      &=a^ib^{j+1}a^{-i}.\notag \\ \notag
\end{align}

\noindent This proves Claim 1.  Note that we have also shown that $b_i^j\in \boldsymbol{\mu}$ for all $i,j\in\{1,\ldots,\nu\}$.  Then by Lemma \ref{L:prodmonad}, $b_1\cdots b_i$ is defined and $b_1\cdots b_i\in \boldsymbol{\mu}$ for all $i\in\{1,\ldots,\nu\}$.

\

\noindent The following claim finishes the proof of the lemma.

\

\noindent Claim 2:  For all $i\in\{1,\ldots,\nu\}$, $(ab)^i$ is defined and $(ab)^i=(b_1\cdots b_i)a^i$.

\

\noindent We prove Claim 2 by internal induction on $i$.  The assertion is clearly true for $i=1$ and now suppose it holds for all $j\in \n^*$ with $j\leq i$.  Suppose $i+1\leq \nu$.  To prove that $(ab)^{i+1}$ is defined, it is enough to check that $((ab)^k,(ab)^l)\in \o^*$ for all $k,l\in\{1,\ldots,i\}$ with $k+l=i+1$.  By the induction hypothesis, $(ab)^k=(b_1\cdots b_k)a^k\sim a^k$ and $(ab)^l=(b_1\cdots b_l)a^l\sim a^l$ and the desired result now follows from the fact that $(a^k,a^l)\in U^* \times U^*$.  Next note that $a^{i+1}b=(a^{i+1}\cdot b)\cdot(a^{-i-1}a^{i+1})=(a^{i+1}ba^{-i-1})a^{i+1}=b_{i+1}a^{i+1}$.  Hence, by the induction hypothesis, we have
\begin{align}
(ab)^{i+1}&=((ab)^i)\cdot(ab)\notag \\
                 &=((b_1\cdots b_i)a^i)\cdot(ab)\notag \\
                 &=(b_1\cdots b_i)\cdot (a^i(ab))\notag \\
                 &=(b_1\cdots b_i)(a^{i+1}b)\notag \\
                 &=(b_1\cdots b_i)(b_{i+1}a^{i+1})\notag \\
                 &=((b_1\cdots b_i)b_{i+1})a^{i+1}\notag \\
                 &=(b_1\cdots b_{i+1})a^{i+1}.\notag \\ \notag
\end{align}

\end{proof}

\

\begin{lemma}\label{L:commutator}
Let $\nu > \n$ and $a\in G(\nu)$ be such that $a^\nu$ is defined and $a^i\in G^*_{\ns}$ for all $i\in \{1,\ldots,\nu\}$.  Suppose also that $b\in \boldsymbol{\mu}$ is such that $b^\nu$ is defined and $b^i\in G^*_{\ns}$ and $a^i \sim b^i$ for all $i\in\{1,\ldots,\nu\}$.  Then $a^{-1}b\in G^{\lilo}(\nu)$.
\end{lemma}

\begin{proof}
We first make a series of reductions.  Note that by Lemma \ref{L:Gsigmagroup}, we have $a^{-1}b\in G(\nu)$.  Let $V$ be a compact symmetric neighborhood of $1$ in $G$ such that $V\subseteq \u_2$.  Then since $\nu=\bigO(\ord_V(a^{-1}b))$, by replacing $\nu$ with $\ord_V(a^{-1}b)$ if $\nu > \ord_V(a^{-1}b)$, we may as well assume that $(a^{-1}b)^i$ is defined and $(a^{-1}b)^i\in G^*_{\ns}$ for all $i\in \{1,\ldots,\nu\}$.  Let $Q=\{1,a,a^{-1},b,b^{-1}\}$.  Suppose that $Q^\nu$ is not defined.  Let $\eta\in \n^*\setminus \n$ be the largest element of $\n^*$ such that $Q^\eta$ is defined.  If $\eta=\lilo(\nu)$, then by Lemma \ref{L:prodmonad}, $Q^i\subseteq \muu$ for all $i\in \{1,\ldots,\eta\}$.  But then, by Lemma \ref{L:EZFacts}, $Q^{\eta+1}$ is defined, a contradiction.  Thus $\nu=O(\eta)$.  So, by replacing $\nu$ by $\eta$ if necessary, we can assume $Q^\nu$ is defined.  If $a\in G^{\lilo}(\nu)$, then by Lemma \ref{L:Gsigmagroup}, $(a^{-1}b)^i$ is defined and $(a^{-1}b)^i\sim a^{-i}\in \muu$ for all $i\in \{1,\ldots,\nu\}$, yielding the desired conclusion.  So we may as well assume that $a\notin G^{\lilo}(\nu)$ and by replacing $\nu$ by a smaller element of its archimedean class, we may as well assume that $a^\nu \notin \muu$.  

\

\noindent Now suppose, towards a contradiction, that there is $j\in\{1,\ldots,\nu\}$ such that $(a^{-1}b)^j \notin \boldsymbol{\mu}$.  Note we must have $\nu =\bigO(j)$.  Choose compact symmetric neighborhoods $\u$ and $\w$ of $1$ in $G$ as in the previous section so that $$a^\nu \in \w^* \setminus \u^*\text{ and }(a^{-1}b)^j\in \w^* \setminus (\u^*)^4.$$  Let $\sigma = \ord_\u(Q)$.  Then $\nu = \bigO(\sigma)$ by Lemma \ref{L:prodmonad}.  Let $\phi=\phi_Q:G^* \rightarrow \r^*$ be the internally continuous function constructed in the previous section.  Then we know that $\phi((a^{-1}b)^j)=0$ and $\epsilon :=\phi(1)>0$ and is not infinitesimal.  Also, note

\begin{align}
\epsilon &\leq \|(b^{-1}a)^j\phi - \phi \| \notag \\
               &\leq j\|(b^{-1}a)\phi - \phi \| \notag \\
               &=j\|a\phi - b\phi \| \notag \\
               &=\|j(a\phi-\phi)-j(b\phi -\phi)\|. \notag \\ \notag
\end{align}           

\noindent We will obtain a contradiction by showing that $\|j(a\phi - \phi) - j(b\phi -\phi)\|<\epsilon.$  By Lemma \ref{L:gleason2}, there is a compact symmetric neighborhood $U\subseteq \u$ of $1$ in $G$ such that for all $k>0$ in $\n^*$, if $a^i \in U^*$ and $b^i \in U^*$ for all $i\in\{1,\ldots,k\}$, then 
$$\|(\frac{j}{k})(a^k\phi - \phi) - j(a\phi - \phi)\|< \frac{\epsilon}{3},$$
$$\|(\frac{j}{k})(b^k\phi - \phi) - j(b\phi - \phi)\|< \frac{\epsilon}{3}.$$

\noindent Let $k=\min\{\ord_U(a),\ord_U(b)\}$.  Then the above equalities hold for this $k$.  Note that $\nu=\bigO(k)$ and hence $j=O(k)$.  Since $k< \nu$, we have $a^k \sim b^k$ and so $$\|(\frac{j}{k})(a^k\phi-\phi)-(\frac{j}{k})(b^k\phi-\phi)\|=\frac{j}{k}\|a^k\phi -b^k\phi\|\sim 0.$$ 
\noindent Combining this fact with the previous inequalities yields $$\|j(a\phi - \phi) - j(b\phi -\phi)\|<\epsilon,$$ giving the desired contradiction.
\end{proof}

\

\begin{thm}
Suppose $\sigma > \n$.  Then
\begin{enumerate}
\item $G(\sigma)$ and $G^{\lilo}(\sigma)$ are normal subgroups of $\boldsymbol{\mu}$;
\item if $a\in G(\sigma)$ and $b\in \boldsymbol{\mu}$, then $aba^{-1}b^{-1}\in G^{\lilo}(\sigma)$;
\item $G(\sigma)/G^{\lilo}(\sigma)$ is abelian.
\end{enumerate}
\end{thm}

\begin{proof}
We have already remarked that $G(\sigma)$ and $G^{\lilo}(\sigma)$ are symmetric and that $G(\sigma)$ is closed under multiplication.  That $G(\sigma)$ is a normal subgroup of $\boldsymbol{\mu}$ is part of the content of Lemma \ref{L:Gsigma properties}, (2).  Lemma \ref{L:Gsigmagroup} implies that $G^{\lilo}(\sigma)$ is a group.  Now suppose $a\in G^{\lilo}(\sigma)$ and $b\in \boldsymbol{\mu}$.  We need $bab^{-1}\in G^{\lilo}(\sigma)$.  One can show by internal induction on $\eta \leq \sigma$ that $(bab^{-1})^\eta$ is defined and equal to $ba^\eta b^{-1}$.  In particular, this shows that $bab^{-1}\in G^{\lilo}(\sigma)$.  This finishes the proof of (1).

\

\noindent Let $a\in G(\sigma)$ and $b\in \boldsymbol{\mu}$.  We need to show that $aba^{-1}b^{-1}\in G^{\lilo}(\sigma)$.  Note that $ba^{-1}b^{-1} \in G(\sigma)$ by (1).  Let $U\subseteq \u_2$ be a compact symmetric neighborhood of $1$.  Choose $\tau \in \{1,\ldots,\sigma\}$ such that $\sigma=\bigO(\tau)$, $a^\tau$ and $(ba^{-1}b^{-1})^\tau$ are defined, and $a^i\in U^*$ for $i\in \{1,\ldots,\tau\}$.  One can prove by internal induction that for $i\in\{1,\ldots,\tau\}$, $(ba^{-1}b^{-1})^i=ba^{-i}b^{-1}$ and hence $a^{-i}\sim (ba^{-1}b^{-1})^i$.  By Lemma \ref{L:commutator}, $aba^{-1}b^{-1}\in G^{\lilo}(\tau)\subseteq G^{\lilo}(\sigma)$.  This finishes (2) and (2) obviously implies (3).
\end{proof}

\

\noindent Observe that if $\x\in L(G)$ and $\sigma>\n$, then $\x(\frac{1}{\sigma})\in G(\sigma)$.  To see this, suppose that $i=\lilo(\sigma)$.  Then $(\x(\frac{1}{\sigma}))^i=\x(\frac{i}{\sigma})\in \muu$.

\

\begin{thm}\label{T:addition}
The map $S:L(G)\rightarrow G(\sigma)/G^{\lilo}(\sigma)$ defined by $$S(\x)=\x(\frac{1}{\sigma})G^{\lilo}(\sigma)$$ is a bijection.  
\end{thm}

\begin{proof}
Suppose $\x,\y\in L(G)$ and $S(\x)=S(\y)$.  Letting  $a:=\x(\frac{1}{\sigma})$ and $b:=\y(\frac{1}{\sigma})$, we have that $a^{-1}b\in G^{\lilo}(\sigma)$.    Let $U$ be a compact symmetric neighborhood of $1$ in $G$ with $U\subseteq \u_2$.  Let $\tau:=\min\{\ord_U(a),\sigma\}$.  By Lemma \ref{L:Gsigmagroup}, we have $(a\cdot (a^{-1}b))^i$ is defined, nearstandard and infinitely close to $a^i$ for all $i\in\{1,\ldots,\tau\}$, i.e. $a^i \sim b^i$ for all $i\in\{1,\ldots,\tau\}$.  Then  for all $i\in\{1,\ldots,\tau\}$ such that $\frac{i}{\sigma}\in \dom(\x)\cap \dom(\y)$, we have $\x(\frac{i}{\sigma})=\y(\frac{i}{\sigma})$.  Since $\sigma=\bigO(\tau)$, this implies that $\x=\y$.  Hence $S$ is 1-1.

\

\noindent Now suppose $b\in G(\sigma) \setminus G^{\lilo}(\sigma)$.  Consider the local 1-ps $X_b$ of $G$ defined by $X_b(t)=\st(b^{[t\sigma]})$ on its domain; see Lemma \ref{L:regelement}.  Let $b_1=X_b(\frac{1}{\sigma})\in \boldsymbol{\mu}$.  Choose $\tau\in\{1,\ldots,\sigma\}$ such that $\sigma = \bigO(\tau)$ and $\frac{\tau}{\sigma}\in \dom(X_b)$.  Note then that $b_1^\tau$ is defined.  For any $i\in\{1,\ldots,\tau\}$, $b^i$ is defined and $b^i \sim b_1^i$.  By Lemma \ref{L:commutator}, $b^{-1}b_1\in G^{\lilo}(\tau)\subseteq G^{\lilo}(\sigma)$ and hence $bG^{\lilo}(\sigma)=b_1G^{\lilo}(\sigma).$    Hence $S$ is onto.
\end{proof}

\

\noindent We use Theorem \ref{T:addition} to equip $L(G)$ with an abelian group operation $+_\sigma$.  More explicitly, for $\x,\y \in L(G)$, we define $\x+_\sigma \y:=S^{-1}(S(\x)\cdot S(\y))$.  By the proof of Theorem \ref{T:addition}, we see that if $t\in \dom(\x+_\sigma \y)$ and $\nu$ is such that $\frac{\nu}{\sigma}\sim t$, then 
$$(\x+_\sigma \y)(t)\sim [\x(\frac{1}{\sigma})\y(\frac{1}{\sigma})]^\nu.$$   

\

\noindent  In ~\cite{H}, a proof is given, in the global setting, that this group operation is independent of $\sigma$.  A local version of that proof is probably possible, but this independence also follows from the Local H5, which we prove without using that  $+_\sigma$ is independent of $\sigma$.  In light of this independence of $\sigma$, we write $\x+\y$ instead of $\x+_\sigma \y$.

\

\noindent \textbf{Existence of Square Roots}  

\

\noindent In the rest of this section, we assume that $G$ is $\Ns$ and that $\u$ is a special neighborhood of $1$ in $G$.  We now lead up to the proof of the crucial fact that the image of the local exponential map is a neighborhood of $1$ in $G$.  In particular, we will see that there is an open neighborhood of $1$ in $G$ which is \textit{ruled by local 1-parameter subgroups} in the sense that every element in this open neighborhood lies on some local 1-ps of $G$.  In the rest of this section, we do not follow \cite{H}, but rather \cite{vdD}, which contains the following series of lemmas in the global setting.

\

\begin{lemma}
Suppose $\sigma > \n$ and $a,b\in G(\sigma)$.  Then $[X_a]+[X_b]=[X_{ab}]$.
\end{lemma}

\begin{proof}
From Theorem \ref{T:addition}, we have 
\begin{align}
S([X_a]+[X_b])&=S([X_a])\cdot S([X_b]) \notag \\
                          &=(aG^{\lilo}(\sigma))(bG^{\lilo}(\sigma)) \notag \\
                          &=(ab)G^{\lilo}(\sigma) \notag \\
                          &=S([X_{ab}]). \notag \\ \notag
\end{align}
\noindent The result now follows since $S$ is injective.
\end{proof}

 \
  
 \noindent The global version of the next lemma is in Singer \cite{S}, but the proof there is unclear. 
 \begin{lemma}\label{L:Singer}
 Suppose $\sigma > \n$ and $a\in G(\sigma)$.  Then $a=bc^2$ with $b\in G^{\lilo}(\sigma)$ and $c\in G(\sigma)$.  
 \end{lemma}
 
 \begin{proof}
Let $c:=X_a(\frac{1}{2\sigma})\in \boldsymbol{\mu}$.  Then $c^{-2}=X_a(-\frac{1}{\sigma})=X_{a^{-1}}(\frac{1}{\sigma})$, which implies that $[X_{c^{-2}}]=[X_{a^{-1}}]$.  From this, we conclude that $$[X_{ac^{-2}}]=[X_{a}]+[X_{c^{-2}}]=[X_1]=\O.$$  Let $b:=ac^{-2}$.  Then since $[X_b]=\O$, we have $b\in G^{\lilo}(\sigma)$.         
 \end{proof}
 
 \
 
 \begin{lemma}\label{L:q}
 There exists $q\in \mathbb{Q}^{>0}$ such that for all $a,b,c\in \boldsymbol{\mu}$, if $a=bc^2$ and $\ord(b)\geq \ord(a)$, then $\ord(c)\geq q\cdot \ord(a)$.
 \end{lemma}
 
 \begin{proof}
 Suppose no such $q$ exists.  Then for all $n\geq 1$, there are $a_n,b_n,c_n\in \boldsymbol{\mu}$ such that $a_n=b_nc_n^2$ with ord$(b_n)\geq \ord(a_n)$ and $n \ord(c_n) < \ord(a_n)$.  By saturation, we have $a,b,c\in \boldsymbol{\mu}$ such that $a=bc^2$, ord$(b)\geq \ord(a)$ and $n \ord(c)<\ord(a)$ for all $n$.  Let $\sigma:=\ord(c)>\n$.  Then $$\O=[X_a]=[X_{bc^2}]=[X_b]+[X_{c^2}]=\O+2[X_c],$$ implying that $[X_c]=\O$, an obvious contradiction.  
 \end{proof}
 
 \
 
 \noindent Fix $q$ as in the above lemma.  For $a,b\in \boldsymbol{\mu}$, we let $[a,b]:=aba^{-1}b^{-1}$, the commutator of $a$ and $b$.  We will frequently use the identity $ab=[a,b]ba$.
 
 \
 
 \begin{lemma}
 Let $a\in \boldsymbol{\mu}$.  Then for all $\nu \geq 1$, there are $b_\nu,c_\nu \in G^*$ such that $b_{\nu}\cdot c_\nu \cdot c_\nu$ is defined, $a=b_\nu c_\nu^2$, $\ord(b_\nu) \geq \nu \ord(a)$ and $\ord(c_\nu)\geq q \ord(a)$. 
 \end{lemma}
 
 \begin{proof}
Note that the assertion is trivial if $a=1$ and so we suppose $a\not=1$.  Next note that if the assertion is true for a given $\nu$, then we have $b_\nu, c_\nu \in\boldsymbol{\mu}$ as they will have infinite order.  By Lemmas \ref{L:Singer} and \ref{L:q}, the assertion is true for some $\nu>\n$.  It is clearly then true for every $\nu'\in\{1,\ldots,\nu\}$.  Hence, if we can prove the assertion for $\nu + 1$, then we will be done by internal induction.  If $b_\nu=1$, then the assumed factorization witnesses the truth of the assertion for all $\nu' $.  We thus assume $b_\nu\not=1$.  By Lemma \ref{L:Singer}, there are $b,c\in \boldsymbol{\mu}$ and $\tau_1>\n$ such that $b_\nu=bc^2$ with $\ord(b)\geq \tau_1 \ord(b_\nu)$.  Note that then $\ord(b)\geq \tau_1 \nu \ord(a)$.  By Lemma \ref{L:q}, we also know that $\ord(c)\geq q \ord(b_\nu)\geq q \nu \ord(a)$.  Since $\boldsymbol{\mu}$ is a group, we have 
\begin{align}
a=bc^2c_\nu^2&=bccc_\nu c_\nu \notag \\
                           &=bc[c,c_\nu]c_\nu cc_\nu \notag \\
                           &=b[c,[c,c_\nu]][c,c_\nu]cc_\nu c c_\nu \notag \\
                           &=b[c,[c,c_\nu]][c,c_\nu](cc_\nu)^2 \notag \\
                           &=b_{\nu+1}(c_{\nu+1})^2, \notag \\ \notag
\end{align}
   
\noindent where $b_{\nu+1}:=b[c,[c,c_\nu]][c,c_\nu]$ and $c_{\nu+1}:=cc_\nu$.  Now note that there is $\tau_2 >\n$ such that 
$$\ord([c,c_\nu])\geq \tau_2 \ord(c)\geq \tau_2 q\nu \ord(a).$$  Also there is $\tau_3> \n$ such that $$\ord([c,[c,c_\nu]])\geq \tau_3 \tau_2 \ord(c) \geq \tau_3 \tau_2 q\nu \ord(a).$$  In combination with the facts that $\ord(b)\geq \tau_1 \nu \ord(a)$ and $G(\tau \ord(a))$ is a group, where $\tau:=\min(\tau_1,\tau_2)$, we have that $\ord(b_{\nu+1})\geq (\nu+1)\ord(a)$.  Finally, Lemma \ref{L:q} yields that $\ord(c_{\nu+1})\geq q\ord(a)$ and so we are finished.       
 \end{proof}
 
 \
 
 \begin{lemma}\label{L:sqrt}
 For every $a\in \boldsymbol{\mu}$, there is $b\in \boldsymbol{\mu}$ with $a=b^2$.
 \end{lemma}
 
 \begin{proof}
 Fix $\nu> \n$.  Then by the previous lemma, the internal set of all $\eta\in \n$ such that for each $x\in V_\eta$ there are $y,z\in \u^*$ such that $y \cdot z \cdot z$ is defined, $x=yz^2$, $\ord(y)\geq \nu \ord(x)$ and $\ord(z)\geq q\ord(x)$ includes all infinite $\eta$.  By underspill, there is some $n>0$ in this set.  Now let $x\in V_n$.  Then there are $y,z\in \u^*$ such that $y\cdot z \cdot z$ is defined, $x=yz^2$ and $\ord(y)\geq \nu \ord(x)$.  Since $\ord(y)$ is infinite, $y\in \boldsymbol{\mu}$.  Hence $x=\st(z)^2$, i.e. every element of $V_n$ has a square root in $\u$.  By transfer, this implies that every $a\in \boldsymbol{\mu}$ has a square root $b\in \u^*$ and $b\in \boldsymbol{\mu}$, else the group $\{1,\st(b)\}$ is contained in $\u$, a contradiction.           
 \end{proof}
 
 \
 
 \begin{lemma}\label{L:ordinf}
 Suppose $a\in \boldsymbol{\mu}$, $b\in G^*$ and $\nu>0$ are such that $\ord(b)\geq \nu$ and $a=b^\nu$.  Then for all $i\in\{1,\ldots,\nu\}$, we have $b^i\in \boldsymbol{\mu}$.
 \end{lemma}
 
 \begin{proof} 
We begin the proof with a claim.

\

\noindent Claim:  For all $n$, $b^{n\nu}$ is defined.

\

\noindent We prove this claim by induction on $n$, the case $n=1$ being true by assumption.  Assume inductively that $b^{n\nu}$ is defined.  We claim that $$b^j\in \u_2^* \text{ for all }j\in \{1,\ldots,n\nu\}.\quad (\dagger)$$  To see this, write $j=n'\nu +j'$, where $n'\in \n$ and $j'\in \{1,\ldots,\nu-1\}$.  Then $$b^j=(b^\nu)^{n'}\cdot b^{j'}\sim b^{j'}\in \u^*.$$  In order to show that $b^{n\nu + \nu}$ is defined, we prove by internal induction on $i$ that $b^{n\nu+i}$ is defined and $b^{n\nu+i}\in \u_2^*$ for all $i\in \{0,\ldots,\nu\}$.  The case $i=0$ is true since $b^{n\nu}=a^n\in\muu$.  Suppose, inductively, that it holds for a given $i$ and that $i+1\leq \nu$.  Let $k,l\in \{1,\ldots,n\nu+i\}$ be such that $k+l=n\nu +i+1$.  Then by $(\dagger)$ and the inductive assumption, $(b^k,b^i)\in \u_2^*\times \u_2^*\subseteq \o^*$.  Hence $b^{n\nu+i+1}$ is defined.  But then $b^{i+1}\in \u^*$ and $b^{n\nu+i+1}\sim b^{i+1}$.  Hence $b^{n\nu+i+1}\in \u_2^*$.  This proves the claim.

\

\noindent Now suppose, towards a contradiction, that there is $i\in\{1,\ldots,\nu-1\}$ such that $b^i\notin \boldsymbol{\mu}$. Let $x:=\st(b^i)\in \u\setminus \{1\}$.  From the claim, we see that if $x^n$ is defined for a given $n$, then $x^n \in \u$.  From this, one can show by induction that for all $n$, $x^n$ is defined and $x^n\in \u$.  Hence $x^\z \subseteq \u$, a contradiction.    
 \end{proof}
 
\
 
 \begin{lemma}
 Given $a\in \boldsymbol{\mu}$ and $\nu$, there is $b\in \boldsymbol{\mu}$ such that $\ord(b)\geq 2^\nu$ and $a=b^{2^\nu}$.
 \end{lemma}
 
 \begin{proof}
 We prove this by internal induction on $\nu$, noting that the assertion is internal since $b\in \boldsymbol{\mu}$ can be replaced by $b\in G^*$ by the previous lemma.  The assertion is true for $\nu=1$ by Lemma \ref{L:sqrt}.  Now suppose the assertion is true for $\nu$, i.e. that $a=b^{2^\nu}$ with $\ord(b)\geq 2^\nu$.  Since $b\in \boldsymbol{\mu}$, there is $c\in \boldsymbol{\mu}$ with $b=c^2$ by Lemma \ref{L:sqrt}.  
 
\

\noindent Claim:  For all $i\in\{1,\ldots,2^\nu\}$, $c^{2i}$ is defined and $(c^2)^i=c^{2i}$.

\

\noindent We prove the claim by internal induction.  It is obviously true for $i=1$ and suppose it is true for a given $i\leq 2^\nu$.  Suppose $i+1\leq 2^\nu$.  We first show $c^{2i+1}$ is defined.  By induction, we know that $c^{2i}$ is defined, and so by Lemma \ref{L:EZFacts}, it is enough to show $(c^k,c^l)\in \o^*$ for $k,l\in\{1,\ldots,2i\}$ with $k+l=2i+1$.  If $k$ is even, then $$c^k=(c^2)^{\frac{k}{2}}=b^{\frac{k}{2}}\in \boldsymbol{\mu}$$ by Lemma \ref{L:ordinf} since $\frac{k}{2}\leq i\leq 2^{\nu}$.  Similarly, if $k$ is odd, then $$c^k=c\cdot (c^2)^{\frac{k-1}{2}}=cb^{\frac{k-1}{2}}\in \boldsymbol{\mu}.$$  Thus, for $k,l\in\{1,\ldots,2i\}$, one has $(c^k,c^l)\in \boldsymbol{\mu} \times \boldsymbol{\mu} \subseteq \o^*$.  Now one can repeat the same argument using that $c^{2i+1}$ is defined to conclude that $c^{2i+2}$ is defined.  Now that we know that $c^{2i+2}$ is defined, we have, by induction, that $(c^2)^{i+1}=(c^2)^i\cdot c^2=c^{2i}\cdot c^2=c^{2i+2}$.  
 
 \
 
 \noindent By the claim, we have that $a=(c^2)^{2^\nu}=c^{2^{\nu+1}}$.  It remains to show that ord$(c)\geq 2^{\nu+1}$.  Let $i\in\{1,\ldots,2^{\nu+1}\}$.  Then if $i$ is even, we have $c^i=b^{\frac{i}{2}}\in \boldsymbol{\mu}$  by Lemma \ref{L:ordinf} since $\frac{i}{2}\leq 2^\nu \leq \ord(b)$.  If $i$ is odd, then $c^i=c\cdot c^{\frac{i-1}{2}}\in \boldsymbol{\mu}$ for the same reason.  Thus $c^i\in \boldsymbol{\mu}$ for all $i\in\{1,\ldots,2^{\nu+1}\}$ and so ord$(c)\geq 2^{\nu+1}$.
 \end{proof}
  
 \
 
 \begin{thm}\label{T:kneigh}
 $K$ is a neighborhood of $1$.
 \end{thm}

\begin{proof}
Fix $\nu> \n$.  Then the internal set of all $\eta$ such that for each $x\in V_\eta$ there is $y\in \u^*$ such that $\ord(y)\geq 2^\nu$ and $x=y^{2^\nu}$ contains all infinite $\eta$.  So by underspill, there is some $n>0$ in this set.  Now given $x\in V_n$, there is $a\in \u^*$ such that $\ord(a)\geq 2^\nu$ and $x=a^{2^\nu}$.  Let $\sigma:=2^\nu$.  Then for $t\in I$, $X_a(t)=\st(a^{[t\sigma]})\in \u$ and $X_a(1)=x$.  Thus $[X_a]\in \k$ and $x=[X_a](1)\in K$.  We have shown $V_n\subseteq K$ and this completes the proof. 
\end{proof}
 
\ 
  
\section{The Space $L(G)$}

\noindent Throughout this section, we assume that our (locally compact) $G$ is $\Ns$.  We fix a special neighborhood $\u$ of $1$ in $G$.  We will equip $L(G)$ with a topology in such a way that it becomes a locally compact finite dimensional real topological vector space, where the operations of scalar multiplication and vector addition are as previously defined.  The section culminates with the important fact that locally compact $\Ns$ local groups are locally euclidean.

\

\noindent \textbf{The Compact-Open Topology}

\

\noindent We first recall that if $X,Y$ are hausdorff spaces, we can turn $\C(X,Y)$, the set of continuous maps from $X$ to $Y$, into a topological space by equipping it with the compact-open topology.  This topology has as a subbasis the sets $B(C,U):=\{f\in \C(X,Y) \ | \ f(C)\subseteq U\}$, where $C$ ranges over all compact subsets of $X$ and $U$ ranges over all open subsets of $Y$.  In the solution of the H5, one equipped the space $\C(\r,G)$ with the compact-open topology and then gave the space of 1-parameter subgroups the corresponding subspace topology.

\

\noindent  Analogously, we will define a topology on $L(G)$ by declaring the subbasic open set to be the sets $B_{C,U}:=\{ \x \in L(G) \ | \ C\subseteq \dom(\x) \text{ and }\x(C)\subseteq U\}$, where $C$ ranges over all compact subsets of $(-2,2)$ and $U$ ranges over all open subsets of $G$.  We next describe the monad structure of $L(G)$ with respect to this topology. 

\

\begin{lemma}\label{L:monadLG}
Let $\x\in L(G)$ and $\y \in L(G)^*$.  Then $\y \in \muu(\x)$ if and only if $\dom(\x)\cap (-2,2)\subseteq \dom(\y)$ and for every $t\in \dom(\x)\cap (-2,2)$ and every $t'\in \muu(t)$, $\y(t')\in \muu(\x(t))$.
\end{lemma}

\begin{proof}
First suppose that $\y \in \muu(\x)$.  Since $\y \in B_{C,G}^*$ for every compact $C\subseteq \dom(\x)\cap (-2,2)$, we have $\dom(\x)\cap (-2,2)\subseteq \dom(\y)$.  Now suppose $t\in \dom(\x)\cap (-2,2)$ and $U$ is an open neighborhood of $\x(t)$ in $G$.  Let $C$ be a compact neighborhood of $t$ with $C\subseteq \dom(\x)\cap (-2,2)$ and such that $\x(C)\subseteq U$.  Since $\y \in B_{C,U}^*$, we have $\y(C^*)\subseteq U^*$ and so in particular $\y(t')\in U^*$.  Thus $\y(t')\in \muu(\x(t))$.  

\

\noindent For the converse, suppose $\dom(\x)\cap (-2,2)\subseteq \dom(\y)$ and for all $t\in \dom(\x)\cap (-2,2)$ and $t'\in \muu(t)$, we have $\y(t')\in \muu(\x(t))$.  Suppose $\x\in B_{C,U}$.  Then for all $t\in C^*$, we have $\y(t)\in \muu(\x(\st(t)))\subseteq U^*$.  This implies that $\y\in B_{C,U}^*$ and thus $\y\in \muu(\x)$.
\end{proof}

\

\noindent In the rest of the paper, we assume the special neighborhood $\u$ of an NSS local group is chosen so small that $\u\subseteq \u_6$.

\

\begin{lemma}\label{L:calkcompact}
$\k$ is a compact neighborhood of $\O$ in $L(G)$.
\end{lemma}

\begin{proof}
To see that $\k$ is a neighborhood of $\O$, choose an open $W\subseteq \u$ with $1\in W$.  Let $\k_W:=\{\x \in L(G) \ | \ I\subseteq \dom(\x) \text{ and }\x(I)\subseteq W\}\subseteq \k$.  Clearly $\O\in \k_W$ and $\k_W$ is an open set in $L(G)$ by the definition of the topology.  

\

\noindent We now prove that $\k$ is compact.  Let $\y \in \k^*$.  It suffices to find $\x \in L(G)$ such that $\y\in \muu(\x)$.  Choose $Y\in \y$ with $I^*\subseteq \dom(Y)$.  Now define $X:I\to G$ by $X(t)=\st(Y(t))$.  By Lemma \ref{L:extension}, we may extend the domain of $X$ to ensure that it is a local 1-ps.  We claim that $\x:=[X]$ is such that $\y\in \muu(\x)$.  We first note that $(-2,2)^* \subseteq \dom(\y)$.  To see this, suppose $t\in (1,2)^*$.  Then $(Y(t-1),Y(1))\in \u^*\times \u^*\subseteq \o^*$, so one can define $Y(t):=Y(t-1)Y(1)$.  One defines $Y(t)$ for $t\in (-2,-1)^*$ in a similar fashion.  An easy check verifies that $Y:(-2,2)^*\to G^*$ is an internal local 1-parameter subgroup.  Furthermore, if $t\in \dom(\x)\cap (1,2)$ and $t'\in \muu(t)$, we have $$\y(t')=\y(t'-1)\y(1)\in \muu(X(t-1)X(1))=\muu(X(t)).$$ A similar argument deals with the case that $t\in \dom(\x)\cap (-2,-1)$, finishing the proof of the lemma. 
\end{proof}

\

\begin{cor}\label{C:Ehomeo} 
$E:\k \rightarrow K$ is a homeomorphism.  
\end{cor}

\begin{proof}
By Lemma \ref{L:monadLG}, it is clear that $E$ is continuous.  The corollary is now immediate from the fact that $E$ is a continuous bijection and $\k$ is compact.
\end{proof}

\

\noindent Suppose $\x \in L(G)$ and $s\in (0,2)$ is such that $[-s,s]\subseteq \dom(\x)$ and $\x([-s,s])\subseteq \u_2$.  Let $U\subseteq \u_2$ be a neighborhood of $1$.  Consider the set $$N_\x(s,U):=\{\y \in L(G) \ | \ \y(t)\in \x(t)U \text{ for all }t\in [-s,s]\}.$$

\

\begin{lemma}
For each neighborhood $U$ of $1$ in $G$ with $U\subseteq \u_2$, $N_\x(s,U)$ is a neighborhood of $\x$ in $L(G)$ and the collection 
$$\{N_\x(s,U) \ | \ U \text{ is a neighborhood of }1 \text{ in }G \text{ and }U\subseteq \u_2\}$$ is a neighborhood base of $\x$ in $L(G)$.
\end{lemma}

\begin{proof}
Let $\y \in \muu(\x)$.  Since $[-s,s]\subseteq \dom(\x)\cap (-2,2)\subseteq \dom(\y)$, we have $[-s,s]^*\subseteq \dom(\y)$.  Now fix $t\in [-s,s]^*$ and $U\subseteq \u_2$ a neighborhood of $1$ in $G$.  Then $\y(t),\x(t)\in \muu(\x(\st(t))$, whence $\x(t)^{-1}\y(t)\in \muu$.  Thus $\y(t)\in \x(t)U^*$ and hence $\y \in N_\x(s,U)^*$.  This implies that $\muu(\x)\subseteq N_\x(s,U)^*$, whence $N_\x(s,U)$ is a neighborhood of $\x$.

\

\noindent Now suppose that $\y \in N_\x(s,U)^*$ for every neighborhood $U$ of $1$ in $G$ with $U\subseteq \u_2$.  We must show $\y\in \muu(\x)$.  We first need to check that $\dom(\x)\cap (-2,2)\subseteq \dom(\y)$.  Say $\dom(\x)\cap (-2,2)=(-r,r)$.  Choose $n>0$ such that $\frac{1}{n}(-r,r)\subseteq [-s,s]$.  Since $\y \in N_\x(s,\u_{2n})^*$, we see that the domain of $\y$ contains $(-r,r)$ since the rule $\y(t)=\y(\frac{t}{n})^n$ for $t\in (-r,r)$ defines an internal local 1-parameter subgroup.  Now fix $t\in (-r,r)$ and $t'\in \muu(t)$.  We must show that $\y(t')\in \muu(\x(t))$.  Let $U$ be a neighborhood of $\x(t)$.  We will show that $\y(t')\in U^*$.  Since $\x(\frac{t}{n})\in \dom(p_n)$, there is an open neighborhood $W_1$ of $\x(\frac{t}{n})$ such that $W_1\subseteq \dom(p_n)$ and $p_n(W_1)\subseteq U$.  By local compactness, we can choose a compact neighborhood $W$ of $1$ in $G$ with $W\subseteq \u_2$ such that $\{\x(\frac{t}{n})\}\times W \subseteq \o$ and $\x(\frac{t}{n})W\subseteq W_1$.  Since $\y \in N_\x(s,W)^*$, we know that $\y(\frac{t'}{n})\in \x(\frac{t'}{n})W^*\subseteq W_1^*$.  Indeed, suppose $a\in W^*$.  Then $\x(\frac{t'}{n})a=\x(\frac{t}{n})\cdot (\x(\frac{t-t'}{n})a)$ and $\x(\frac{t-t'}{n})a\sim a$, whence $\st(\x(\frac{t-t'}{n})a)\in W$.  Since $W_1$ is open, we see that $\x(\frac{t'}{n})a\in W_1^*$.  We now see that $\y(\frac{t'}{n})\in \dom(p_n)^*$ and $\y(t')=\y(\frac{t'}{n})^n \in U^*$.
\end{proof}

\

\begin{thm}\label{T:L(G)vs}
$L(G)$ is a locally compact, finite dimensional real topological vector space.
\end{thm}

\begin{proof}
If we can prove that $L(G)$ is a topological vector space, then since it is locally compact by Lemma \ref{L:calkcompact}, it is finite-dimensional by a theorem of Riesz.  We first prove that scalar multiplication is continuous.  Let $r \in \r\setminus \{0\}$ and let $f_r:L(G)\rightarrow L(G)$ be defined by $f_r(\x)=r \cdot \x$.  Suppose $r'\in \muu(r)$ and $\x'\in \muu(\x)$.  We must show that $r'\x'\in \muu(r\x)$.  Since $\dom(\x)\cap (-2,2)\subseteq \dom(\x')$, it is easy to see that $\dom(r\x)\cap (-2,2)\subseteq \dom(r'\x')$.  Furthermore, for $t\in \dom(r\x)\cap (-2,2)$ and $t' \in \muu(t)$, one has $$(r'\x')(t')=\x'(r't')\in \muu(\x(rt))=\muu((r\x)(t)),$$ whence $r'\x'\in \muu(r\x)$.  

\

\noindent We next prove that $+$ is continuous.  We first show that $+$ is continuous at $(\O,\O)$.  To do this, it suffices to show that for any subbasic open set $B_{C,U}$ of $L(G)$ containing $\O$ (so that $1\in U$), there is a subbasic open set $B_{C,W}$ of $L(G)$ containing $\O$ such that $B_{C,W}+B_{C,W}\subseteq B_{C,U}$.  Let $U_1$ be a compact neighborhood of $1$ such that $U_1\subseteq U$.  Let $Z \subseteq \boldsymbol{\mu}$ be an internally open set.  Then by Lemma \ref{L:Gsigmagroup}, we know that for every $a,b \in Z$ and for every $\sigma$ such that for all $i\in\{1,\ldots,\sigma\}$, $a^i$ and $b^i$ are defined and $a^i,b^i\in Z$, we have $(ab)^\sigma$ is defined and $(ab)^\sigma \in U_1^*$.  Hence, by transfer, there is an open set $W$ such that for all $a,b \in W$ and all $n$ with $a^i$ and $b^i$ defined and $a^i,b^i \in W$ for all $i\in\{1,\ldots,n\}$, we have $(ab)^n$ defined and $(ab)^n\in U_1$.  We claim that this is the desired $W$.  Suppose $\x,\y \in B_{C,W}$.  Let $t\in C$ and suppose $\nu$ is such that $\frac{\nu}{\sigma}\sim t$.  Then $(\x+\y)(t)\sim [\x(\frac{1}{\sigma})\y(\frac{1}{\sigma})]^\nu$.  Let $a:=\x(\frac{1}{\sigma})$ and $b:=\y(\frac{1}{\sigma})$.  By assumption, for all $i\in\{1,\ldots,\nu\}$, $a^i$ and $b^i$ are defined and $a^i,b^i \in W^*$.  Thus, $(ab)^\nu \in U_1^*$, implying that  $(\x+\y)(t)\in U_1\subseteq U$ and so $\x+\y \in B_{C,U}$.

\

\noindent  In order to finish the proof that $+$ is continuous, it is enough to show that, for any $\y \in L(G)$, the map $\x \mapsto \x+\y:L(G)\rightarrow L(G)$ is continuous at $\O$ and the map $\x \mapsto \x-\y:L(G)\rightarrow L(G)$ is continuous at $\y$.  Let $W\subseteq L(G)$ be a neighborhood of $\y$ in $L(G)$.  Fix $s\in \dom(\y)\cap (-2,2)$ such that $\y([-s,s])\subseteq \u_2$.  Fix a compact neighborhood $U$ of $1$ in $G$ with $U\subseteq \u_2$ so that $N_\y(s,U)\subseteq W$.  Suppose $X\in \muu(\O)\subseteq L(G)^*$.  By Lemma 6.2, for all $i\leq [s\sigma]$ we have $(\x(\frac{1}{\sigma})\y(\frac{1}{\sigma}))^i$ is defined and infinitely close to $\y(\frac{i}{\sigma})$.  In particular, this means that for every $\x \in \muu(\O)$, whenever $i\leq [s\sigma]$, we have $((\x(\frac{1}{\sigma})\y(\frac{1}{\sigma}))^i$ is defined and 
$$(\x(\frac{1}{\sigma})\y(\frac{1}{\sigma}))^i\in Y(\frac{i}{\sigma})U^*.$$  By overspill, this gives a neighborhood $V$ of $\O$ in $L(G)$ such that for all $\x \in V$ we have $((\x(\frac{1}{\sigma})\y(\frac{1}{\sigma}))^i$ is defined and $$(\x(\frac{1}{\sigma})\y(\frac{1}{\sigma}))^i\in Y(\frac{i}{\sigma})U^*$$ whenever $i\leq [s\sigma]$.  It now follows that for $\x \in V$, we have $$t\in \dom(\x+\y) \text{ and }(\x+\y)(t)\in \y(t)U$$ for all $t\in [-s,s]$, whence $\x+\y \in N_\y(s,U)\subseteq W$.  This proves the continuity of $\x \mapsto \x+\y$ at $\O$.

\

\noindent The continuity of $\x\mapsto \x-\y$ at $\y$ is proved in a similar fashion.  Let $W\subseteq L(G)$ be a neighborhood of $\O$.  Fix $s\in \dom(\y)\cap (-2,2)$.  Choose a compact neighborhood $U$ of $1$ in $G$ with $U\subseteq \u_2$ so that $N_\O(s,U)\subseteq W$.  Suppose $\x\in \muu(\y)\subseteq L(G)^*$.  By Lemma 6.3, $(\x(\frac{1}{\sigma})\y(-\frac{1}{\sigma}))^i$ is defined and infinitesimal whenever $i\leq [s\sigma]$.  In particular, 
$$(\x(\frac{1}{\sigma})\y(-\frac{1}{\sigma}))^i\in U^*$$ whenever $i\leq [s\sigma]$.  By overspill, this gives a neighborhood $V$ of $\y$ in $L(G)$ such that for all $\x \in V$, we have $(\x(\frac{1}{\sigma})\y(-\frac{1}{\sigma}))^i$ is defined and $$(\x(\frac{1}{\sigma})\y(-\frac{1}{\sigma}))^i\in U^*$$ whenever $i\leq [s\sigma]$.  It now follows that for $\x \in V$, we have $$t\in \dom(\x-\y) \text{ and }(\x-\y)(t)\in U$$ for all $t\in [-s,s]$, whence $\x-\y \in N_\O(s,U)\subseteq W$.     

\

\noindent We now need to prove that $L(G)$ is a vector space.  There is only one vector space axiom which is not trivial to verify, namely that $r \cdot (\x+\y)=r \cdot \x + r \cdot \y$.  We first consider the case that $r \in \z$.  Let $a:=\x(\frac{1}{\sigma})$ and $b:=\y(\frac{1}{\sigma})$.  So $S(r \cdot \x + r \cdot \y)=(a ^ r b^r) G^{\lilo}(\sigma)$.  Meanwhile, $S(r \cdot (\x+\y))=((ab)^r)G^{\lilo}(\sigma)$.  But $G(\sigma)/G^{\lilo}(\sigma)$ is abelian, so $$S(r \cdot (\x+\y))=((ab)^r)G^{\lilo}(\sigma)=(a ^ r b^r) G^{\lilo}(\sigma)=S(r \cdot \x + r \cdot \y).$$  We have thus proven this axiom for the case $r \in \z$.

\

\noindent Once again suppose $r \in \z$, $r\not=0$.  We now know that $r \cdot (\frac{1}{r}\x + \frac{1}{r}Y)=\x + \y$, i.e. that $\frac{1}{r}\cdot (\x+\y)=\frac{1}{r}\x+\frac{1}{r}\y$.  So for $r \in \mathbb{Q}$, we have $r \cdot (\x+\y)=r \cdot \x + r \cdot \y$.  We finally conclude that the axiom holds for all $r \in \r$ from the fact that the operations of addition and scalar multiplication are continuous.             
\end{proof}

\

\begin{rmk}
Suppose $H$ and $H'$ are locally compact pure local groups.  Then the proof of Theorem \ref{T:L(G)vs} shows that $L(H)$ is a real topological vector space.  The $\Ns$ assumption was only used to obtain local compactness, and hence finite-dimensionality.  Moreover, given any local morphism $f:H\rightarrow H'$, one can verify that the map $L(f):L(H)\rightarrow L(H')$ is $\r$-linear and continuous.
\end{rmk}

\

\begin{cor}
$G$ is locally euclidean.
\end{cor}

\begin{proof}
Since $L(G)$ is a finite dimensional topological vector space over $\r$, we have an isomorphism $L(G)\cong \r^n$ of real vector spaces for some $n$, and any such isomorphism must be a homeomorphism.  Hence, by Lemma \ref{L:calkcompact}, $\k$ contains an open neighborhood $U$ of $1$ which is homeomorphic to an open subset of $\r^n$.   Shrink $U$ if necessary so that $E(U)\subseteq W$, where $W$ is an open neighborhood of $1$ in $G$ contained in $K$; such a $W$ exists by Theorem \ref{T:kneigh}.  Then by Corollary \ref{C:Ehomeo}, $E(U)$ is an open neighborhood of $1$ in $G$ homeomorphic to an open subset of $\r^n$.
\end{proof}

\

\section{Local H5 for $\Ns$ local groups}

\

\noindent The goal of this section is to prove the Local H5 for $\Ns$ local groups using the local version of the Adjoint Representation Theorem.  

\

\noindent In this section we assume $G$ is $\Ns$.  Recall that our special neighborhood has been chosen so that $\u\subseteq \u_6$.

\

\noindent \textbf{Local Adjoint Representation Theorem}

\

\noindent Fix $g\in \u_6$.  Let $\x \in L(G)$ and $X\in \x$.  Choose  $r\in \dom(X)\cap \r^{>0}$ such that $X((-r,r))\subseteq \u_6$.  Define the function $gXg^{-1}:(-r,r)\rightarrow G$ by the rule $(gXg^{-1})(t)=gX(t)g^{-1}$.  We claim that $gXg^{-1}$ is a local 1-ps of $G$.  Certainly $gXg^{-1}$ is continuous.  Suppose $s,t,s+t\in (-r,r)$.  Then since $g,g^{-1},X(t),X(s)\in \u_6$, we have 

\begin{align}
(gXg^{-1})(t+s)&=g(X(t)X(s))g^{-1} \notag \\ \notag
                          &=(gX(t)g^{-1})(gX(s)g^{-1})\\ \notag
                          &=(gXg^{-1})(t)\cdot (gXg^{-1})(s)\\ \notag \\ \notag
\end{align}
Thus $gXg^{-1}$ is a local 1-ps of $G$ and $\a_g(\x):=g\x g^{-1}:=[gXg^{-1}]\in L(G)$ is easily seen to be independent of the choice of $X\in \x$ and $r\in \dom(X)$ as chosen above.    

\

\begin{lemma}
Suppose $\sigma > \n$, $a\in G(\sigma)$ and $g\in \u_6$.  Then  
\begin{enumerate}
\item $gag^{-1}\in G(\sigma)$;
\item $\a_g([X_a])=[X_{gag^{-1}}]$.
\end{enumerate}
\end{lemma}

\begin{proof}
Let $\tau:=\ord(a)$.  Note that $\sigma=\bigO(\tau)$.  

\

\noindent Claim:  If $i\in \{1,\ldots,\tau\}$, then $(gag^{-1})^i$ is defined and $(gag^{-1})^i=ga^ig^{-1}$.

\

\noindent We prove this by internal induction on $i$.  This is certainly true for $i=1$ and we now assume that it holds for a given $i$ and that $i+1\leq \tau$.  Since we have assumed that $(gag^{-1})^i$ is defined, to show $(gag^{-1})^{i+1}$ is defined, it is enough to show, by Lemma \ref{L:EZFacts}, that $((gag^{-1})^k,(gag^{-1})^l)\in \o^*$ for all $k,l\in \{1,\ldots,i\}$ such that $k+l=i+1$.  However, by induction, $(gag^{-1})^k=ga^kg^{-1}$ and $(gag^{-1})^l=ga^lg^{-1}$, and since $g,a^k,a^l\in \u_6^*$, we have $(ga^kg^{-1},ga^lg^{-1})\in \o^*$.  We thus know that $(gag^{-1})^{i+1}$ is defined.  Now $(gag^{-1})^{i+1}=(gag^{-1})^i(gag^{-1})=(ga^ig^{-1})(gag^{-1})=ga^{i+1}g^{-1}$ since $g,a^i,a,g^{-1}\in \u_6^*$.

\

\noindent By the claim, if $i=\lilo(\sigma)$, then $(gag^{-1})^i$ is defined and $(gag^{-1})^i=ga^ig^{-1}\in \muu$ since $a\in G(\sigma)$.  This proves (1).

\

\noindent  We now must prove (2).  Fix $r\in \r^{>0}$ such that $[r\sigma]\leq \tau$, $X_a((-r,r))\subseteq \u_6$ and $gX_a((-r,r))g^{-1}\subseteq \u$.  Then for $t\in (-r,r)$, we have 
\begin{align}
(\a_g([X_a])(t))&=gX_a(t)g^{-1} \notag \\ \notag
                          &=g(\st(a^{[t\sigma]}))g^{-1}\\ \notag
                          &=\st(ga^{[t\sigma]}g^{-1})\\ \notag
                          &=\st((gag^{-1})^{[t\sigma]})\\ \notag
                          &=X_{gag^{-1}}(t),\\ \notag \\ \notag
\end{align}

which finishes the proof.
\end{proof}

\

\noindent The proof of the following lemma is routine.

\

\begin{lemma}
For $g\in \u_6$, $\a_g:L(G)\rightarrow L(G)$ is a vector space automorphism with inverse $\a_{g^{-1}}$.
\end{lemma}

\

\noindent Let $\Aut(L(G))$ denote the group of vector space automorphisms of $L(G)$.  Suppose $\dim_\r(L(G))=n$.  Take an $\r$-linear isomorphism $L(G) \cong \r^n$; it induces a group isomorphism
$\Aut(L(G)) \cong \text{GL}_n(\r) \subseteq \r^{n^2}$, and we take the
topology on $\Aut(L(G))$ that makes this group isomorphism a homeomorphism.
(This topology does not depend on the choice of $\r$-linear isomorphism
$L(G) \cong \r^n$.)  For $T\in \Aut(L(G))^*$, we see that $T$ is nearstandard if $T(\x)\in L(G)^*_{\ns}$ for all $\x\in L(G)$.  For $T,T'\in \Aut(L(G))^*_{\ns}$, we see that $T\sim T'$ if and only if $T(\x)\sim T'(\x)$ for all $\x\in L(G)$.

\

\begin{thm}(Local Adjoint Representation Theorem) We have a morphism $\a:G|\u_6\rightarrow \Aut(L(G))$ of local groups given by $\a(g):=\a_g$.
\end{thm}

\begin{proof}
Suppose $g,h,gh\in \u_6$ and $[X_a]\in L(G)$.  Then
\begin{align}
\a_{gh}([X_a])&=[X_{(gh)a(gh)^{-1}}] \notag \\ \notag 
                        &=[X_{g(hah^{-1})g^{-1}}] \\ \notag
                        &=\a_g(\a_h([X_a])). \notag \\ \notag
\end{align}
Thus $\a(gh)=\a(g)\circ \a(h)$.  All that remains to prove is that $\a$ is continuous.  To do this, it suffices to prove that $\a$ is continuous at $1$.  Suppose $a\in \boldsymbol{\mu}$ and $\x \in L(G)$.
 
\

\noindent Claim:  $\dom(\x)=\dom(\a_a(\x))\cap \r$.

\

\noindent In order to verify the claim, suppose $(-r,r)\subseteq \dom(\x)$.  We must verify that if $s,t,s+t\in (-r,r)$, then $(a\x(s)a^{-1},a\x(t)a^{-1})\in \o^*$ and $(a\x(s)a^{-1})(a\x(t)a^{-1})=a\x(s+t)a^{-1}$.  From the fact that $(\x(s),\x(t))\in \o$ and $a\in \boldsymbol{\mu}$, we have that $(a\x(s)a^{-1},a\x(t)a^{-1})\in \o^*$ and we have that $(a\x(s)a^{-1})(a\x(t)a^{-1})=a\x(s+t)a^{-1}$ from the usual calculations involving associativity when working with infinitesimals and nearstandard elements.

\  

\noindent Now notice that if $t\in \dom(\x)$, then $(\a_a(\x))(t)=a\x(t)a^{-1}\sim \x(t)$.  So by Lemma \ref{L:monadLG}, $\a_a(\x)\sim \x$.  By the nonstandard characterization of the topology on $\Aut(L(G))$, we have $\a_a\sim \id_{L(G)}$, and hence $\a$ is continuous.  
\end{proof}

\

\noindent \textbf{Some Facts About Local Lie Groups}

\

\noindent In the global setting, the Adjoint Representation Theorem and some elementary Lie group theory imply that locally compact $\Ns$ groups are Lie groups.  In the local group setting, we have to work a little harder.  We need a few ways of obtaining local Lie groups and the results that follow serve this purpose.  

\

\begin{df}
A local group $G$ is \textbf{abelian} if there is a neighborhood $U$ of $1$ in $G$ such that $U\subseteq \u_2$ and $ab=ba$ for all $a,b\in U$.
\end{df}

\

\begin{thm}\label{T:abelian}
Suppose $G$ is abelian.  Then $G$ is locally isomorphic to a Lie group.
\end{thm}

\begin{proof}
Suppose we have chosen our special neighborhood $\u$ so small that elements of $\u$ commute with each other.  Recall that we are still assuming that $\u \subseteq \u_6$.

\

\noindent Claim:  Suppose $a,b\in \u$ are such that $(ab)^n,a^n$ and $b^n$ are all defined and $a^i,b^i\in \u$ for all $i\in \{1,\ldots,n\}$.  Then $(ab)^n=a^nb^n$.

\

\noindent We prove the claim by induction on $n$.  The case $n=1$ is trivial.  Now suppose $(ab)^{n+1}$ is defined as are $a^{n+1}$ and $b^{n+1}$.  Also suppose $a^i$ and $b^i$ are in $\u$ for all $i\in \{1,\ldots,n+1\}$.  Then $$(ab)^{n+1}=(ab)^n\cdot (ab)=(a^nb^n)\cdot ab=a^nb^nab=a^{n+1}b^{n+1}.$$

\noindent This proves the claim.

\

\noindent Choose a symmetric open neighborhood $\mathcal{V}$ of $\O$ in $L(G)$ such that $\mathcal{V}\subseteq \k$ and $\x+\y\in \k$ for all $\x,\y \in \mathcal{V}$.  By the transfer of the claim, if $\x,\y\in \mathcal{V}$, then 

\begin{align}
(\x+\y)(1)&=\st((\x(\frac{1}{\sigma})\y(\frac{1}{\sigma}))^\sigma) \notag \\ \notag
                &=\st(\x(\frac{1}{\sigma})^\sigma \y(\frac{1}{\sigma})^\sigma) \notag \\
                &=\x(1)\y(1). \notag \\ \notag
\end{align}

\noindent After possibly shrinking $\mathcal{V}$, we can choose a symmetric open neighborhood $V$ of $1$ in $G$ such that $E(\mathcal{V})=V$.  This witnesses that the equivalence class of the local exponential map is a local isomorphism from $L(G)$ to $G$.  Since $L(G)$ is a Lie group, we are done. 
\end{proof}

\

\noindent To prove the next lemma, we need the following well-known theorem.

\

\begin{thm}(von Neumann, ~\cite{MZ}, pg. 82)
If $H$ is a hausdorff topological group which admits an injective continuous homomorphism into $\GL_n(\r)$ for some $n$, then $H$ is a Lie group.
\end{thm}

\

\begin{lemma}\label{L:embedding}
Suppose $f:G\rightarrow \GL_n(\r)$ is an injective morphism of local groups.  Then $G$ is a local Lie group.
\end{lemma}

\begin{proof}
Let $G'=f(G)\subseteq \text{GL}_n(\r)$.  Let $H$ be the subgroup of $\text{GL}_n(\r)$ generated by $G'$.  Let $\mathcal{F}'$ be the filter of neighborhoods of $1$ in $G$ and $\mathcal{F}''$ be the image of $\mathcal{F}'$ under $f$.  Finally, let $\mathcal{F}$ be the filter in $H$ generated by $\mathcal{F}''$.  It is routine to verify that $\mathcal{F}$ satisfies the properties needed to make it a neighborhood filter at $1$ for a topology on $H$ which makes $H$ a topological group.  By the definition of this topology, the inclusion map $H\hookrightarrow \text{GL}_n(\r)$ is continuous.  By von-Neumann's theorem, $H$ is a Lie group.  The result now follows from the fact that $G$ is homeomorphic with $f(G)$ when $f(G)$ is given the induced topology from $H$.
\end{proof}

\

\noindent We use the following theorem of Kuranishi \cite{K} to finish our proof of the $\Ns$ case.  

\begin{thm}
Let $G$ be a locally compact local group and let $H$ be a normal sublocal group of $G$.  Consider the local coset space $(G/H)_W$ as in Lemma \ref{L:localcoset}.  Suppose 
\begin{enumerate}
\item $H$ is an abelian local Lie group;
\item $(G/H)_W$ is a local Lie group;
\item there is a set $M\subseteq W$ containing $1$ and $W'\subseteq W$, an open neighborhood of $1$ in $G$, such that for every $(zH)\cap W\in \pi(W')$, there is exactly one $a\in M$ such that $a\in (zH)\cap W$; moreover, the map $\pi(W')\rightarrow M$ that assigns to each element of $\pi(W')$ the corresponding $a$ in $M$ is continuous.
\end{enumerate}
Then $G|W$ is a local Lie group.
\end{thm}

\

\noindent \textbf{Proof of the $\Ns$ Case}

\

\noindent We can now finish the proof that the Local H5 holds for $\Ns$ local groups.  Let $G':=G|\u_6$.  Recall the morphism of local groups $\a:G' \rightarrow \Aut(L(G))$.   Then $H:=\ker(\a)$ is a normal sublocal group of $G'$.  Since $H$ is also a locally compact $\Ns$ local group, we know that $H$ must have an open (in $H$) neighborhood $V$ of $1$ ruled by local 1-parameter subgroups of $H$, which are also local 1-parameter subgroups of $G$.  We claim $gh=hg$ for all $g,h\in V$, implying that $H$ is abelian.  Note that if $\u'\subseteq \u$ is a special neighborhood for $H$ and if $\x\in L(H)$ is such that $I\subseteq \dom(\x)$ and $\x(I)\subseteq \u'$, then $I\subseteq \dom(g\x g^{-1})$.  So for $g,h\in V$, writing $h=\x(1)$ for some $\x\in L(H)$, we have $ghg^{-1}=h$.

\

\noindent   By Theorem \ref{T:abelian}, we have that a restriction of $H$ is a local Lie group.  Let us abuse notation and denote this restriction, which is an equivalent sublocal group of $G'$, by $H$.  Note that if we let $G'':=(G'/H)_W$ be some local coset space, then the adjoint representation induces an injective morphism $G''\rightarrow \Aut(L(G))$.  Thus $G''$ is a local Lie group by Lemma \ref{L:embedding}.  In order to finish the proof, we need to show that condition (3) of Kuranishi's theorem is satisfied.  

\

\noindent Since $G''$ is a local Lie group, we can introduce \textit{canonical coordinates of the second kind}.  More precisely, we can find an open neighborhood $\u'$ of $1$ in $G''$ and a basis $\x_1',\ldots,\x_r'$ of $L(G'')$ such that every element of $\u'$ is of the form $\x_1'(s_1)\cdots \x_r'(s_r)$ for a \textit{unique} tuple $(s_1,\ldots,s_r)\in [-\beta',\beta']^r$.  Without loss of generality, we can suppose the closure of $\u'$ is a special neighborhood of $G''$.  Choose a special neighborhood $\u$ of $G'$ such that $\pi(\u)\subseteq \u'$.  Let $Z\subseteq \u$ be an open neighborhood of $1$ ruled by local 1-parameter subgroups of $G'$.  Fix $s_0\in (0,\beta')$ such that $\x_1'(s_0),\ldots,\x_r'(s_0)\in \pi(Z)$.  Choose $x_i\in Z$ such that $\pi(x_i)=\x_i'(s_0)$.  Let $\x_i\in L(G')$ be such that $\x_i(s_0)=x_i$.  

\

\noindent Let $\beta < s_0$ be so that $\x_i(s)\in \u_{2r}$ if $|s|\leq \beta$ and $\x_1(s_1)\cdots \x_r(s_r)\in W$ if $|s_i|\leq \beta$ for all $i\in \{1,\ldots,r\}$.  A uniqueness of root argument yields that $\pi(\x_i(s))=\x_i'(s)$ for $|s|\leq \beta$.  Set $$M:=\{\x_1(s_1)\cdots \x_r(s_r) \ | \ |s_i|\leq \beta \text{ for all }i=1,\ldots,r\}$$ and let $W'\subseteq W$ be an open neighborhood of $1$ in $G$ contained in the image of the map $$(s_1,\ldots,s_r)\mapsto \x_1(s_1)\cdots \x_r(s_r):[-\beta,\beta]^r\rightarrow G.$$  One can check that the choices of $M$ and $W'$ fulfill condition (3) of Kuranishi's theorem.

\begin{flushright}
 $\Box$
\end{flushright}
\

\section{Locally Euclidean local groups are $\Ns$}  

\

\noindent In order to finish the proof of the Local H5, we must now prove that every locally euclidean local group is $\Ns$.  This proceeds in two stages:  locally euclidean local groups are $\Nsc$ and locally connected $\Nsc$ local groups are $\Ns$.

\

\noindent \textbf{Locally euclidean local groups are $\Nsc$} 

\

\noindent We will need the following facts.

\begin{thm}\label{T:compOPS}(\cite{MZ}, pg. 105)
Every compact connected nontrivial hausdorff topological group has a nontrivial 1-parameter subgroup.
\end{thm}

\

\begin{lemma}\label{L:OPSfacts}
Let $H$ be a topological group and $X:\r \rightarrow H$ a 1-parameter subgroup.  
\begin{enumerate}
\item If $H_1$ is a closed subgroup of $H$, then either $X(\r)\subseteq H_1$ or there is a neighborhood $D$ of $0$ in $\r$ such that $X(D)\cap H_1=\{1\}$.  
\item If $X$ is nontrivial, then there is a neighborhood $D$ of $0$ in $\r$ on which $X$ is injective.
\end{enumerate}
\end{lemma}

\begin{proof}
This is immediate from the well-known fact that a closed subgroup of the additive group of $\r$ different from $\{0\}$ and $\r$ is of the form $\z r$ with $r\in \r^{>0}$.
\end{proof}

\

\noindent The following lemma is the local group version of Theorem 5.1 in ~\cite{H}.  

\ 

\begin{lemma}\label{L:neighandgroup}
Suppose $V$ is a neighborhood of $1$ in $G$.  Then $V$ contains a compact subgroup $H$ of $G$ and a neighborhood $W$ of $1$ in $G$ such that every subgroup of $G$ contained in $W$ is contained in $H$.
\end{lemma}

\begin{proof}
Let $W$ be an internal neighborhood of $1$ in $G^*$ such that $W\subseteq \boldsymbol{\mu}$.  Note that if $E_1,\ldots,E_\nu$ is an internal sequence of internal subgroups of $G^*$ contained in $W$ and $a_1,\ldots,a_\nu$ is an internal sequence such that $a_i \in E_i$ for all $i\in \{1,\ldots,\nu\}$, then $a_1\cdots a_\nu$ is defined by Lemma \ref{L:prodmonad}.  We let $S$ be the set of all products $a_1\cdots a_\nu$, where $E_1,\ldots,E_\nu$ is an internal sequence of internal subgroups of $G^*$ contained in $W$ and $a_1,\ldots,a_\nu$ is an internal sequence such that $a_i\in E_i$ for all $i\in \{1,\ldots,\nu\}$.  By Lemma \ref{L:prodmonad}, $S$ is an internal subgroup of $G^*$, $S\subseteq \boldsymbol{\mu}$, and every internal subgroup of $G^*$ contained in $W$ is contained in $S$.  Furthermore, if $H$ is the internal closure of $S$ in $G^*$, then $H$ is an internally compact internal subgroup of $G^*$ containing all of the subgroups of $W$ and $H\subseteq V^*$.  The desired result follows by transfer.
\end{proof}

\

\begin{df}
Following Kaplansky \cite{Ka}, we call a topological space \textbf{feebly finite-dimensional} if, for some $n$, it does not contain a homeomorphic copy of $[0,1]^n$.  
\end{df}

\

\noindent Clearly locally euclidean local groups are feebly finite-dimensional.

\

\begin{lemma}\label{L:eucNSCS}
If $G$ is feebly finite-dimensional, then $G$ is $\Nsc$.
\end{lemma}

\begin{proof}
Suppose that $G$ is feebly finite-dimensional but, towards a contradiction, that $G$ is not $\Nsc$.  We claim that for every compact symmetric neighborhood $U$ of $1$ in $G$ contained in $\u_2$ and for every $n$, there is a compact subgroup of $G$ contained in $U$ which contains a homeomorphic copy of $[0,1]^n$.  Assume this holds for a given $n$ and let $U$ be given.  Lemma \ref{L:neighandgroup} implies that there is a compact symmetric neighborhood $V$ of $1$ in $G$ containing $1$ and a compact subgroup $H\subseteq U$ that contains every subgroup of $G$ contained in $V$.  Since, by assumption, $V$ contains a nontrivial connected compact subgroup of $G$, we have a nontrivial 1-parameter subgroup $X$ of $H$.  By shrinking $V$ if necessary, we can suppose $X(\r)\nsubseteq V$.  By assumption, we have a compact subgroup $G(V)\subseteq V$ of $G$ and a homeomorphism $Y:[0,1]^n \rightarrow Y([0,1]^n)\subseteq G(V)$.  After replacing $X$ by $rX$ for suitable $r\in (0,1)$, we can assume $X([0,1])\subseteq V$, $X$ is injective on $[0,1]$, and $X([0,1])\cap G(V)=\{1\}$.  Since $H$ is a group, we can define $Z:[0,1]\times [0,1]^n \rightarrow H$ by $Z(s,t)=X(s)Y(t)$, which is clearly continuous.  Suppose $Z(s,t)=Z(s',t')$ with $s\geq s'$.  Then we have $X(s-s')=Y(t')Y(t)^{-1}\in X([0,1])\cap G(V)=\{1\}$.  Since $X$ is injective on $[0,1]$, we have $s=s'$, and since $Y$ is injective, we have $t=t'$.  It follows that $Z$ is injective, and thus a homeomorphism onto its image.  
\end{proof}

\

\noindent \textbf{Locally connected $\Nsc$ local groups are $\Ns$}

\

\noindent The next lemma is the local analogue of Lemma 5.3 and Corollary 5.4 in \cite{H}.  The proof there has a gap, which can be filled by changing the hypotheses as in our lemma.

\ 

\begin{lemma}\label{L:lift}
Suppose $H$ is a normal sublocal group of $G$ which is totally disconnected.  Let $\pi:G\rightarrow G/H$ be the canonical projection.  Then $L(\pi)$ is surjective.  
\end{lemma}

\begin{proof}
Let $G':=(G/H)_W$, where $W$ is as in Lemma \ref{L:localcoset}, and let $\y \in L(G')$.  We seek $\x \in L(G)$ so that $\pi \circ \x =\y$.  If $\y$ is trivial, this is obvious.  Thus we may assume, without loss of generality, that $1\in \dom(\y)$ and $\y(1)\not=1_{G'}$.  Fix $\nu > \n$ and let $h:=\y(\frac{1}{\nu})\in \boldsymbol{\mu}(1_{G'})$.  Take a compact symmetric neighborhood $V$ of $1_{G'}$ in $G'$ such that $\y(1)\notin V$.  Take a compact symmetric neighborhood $U$ of $1$ in $G$ with $U\subseteq W$ such that $\pi(U)\subseteq V$.  Since $\pi$ is an open map, we have $\boldsymbol{\mu}(1_{G'})\subseteq \pi(\boldsymbol{\mu})$.  Choose $a\in \boldsymbol{\mu}$ with $\pi(a)=h$.  Let $\sigma:=\ord_U(a)$.  If $\nu \leq \sigma$, then $\pi(a^\nu)\in V^*$, contradicting the fact that $\pi(a^\nu)=h^\nu=1_{G'}\notin V$.  Thus we must have $\sigma < \nu$.  If $i=\lilo(\sigma)$, then $\pi(\st(a^i))=\st(h^i)=1_{G'}$, so $G_U(a)=\{\st(a^i) \ | \ i=\lilo(\sigma)\}$ is a connected subgroup of $G$ contained in $H$.  Since $H$ is totally disconnected, we must have $G_U(a)=\{1_G\}$, i.e. $a\in G(\sigma)$.  Since $a\notin G^{\lilo}(\sigma)$, we have $[X_a]\not= \O$.  Suppose $\sigma =\lilo(\nu)$ and let $t\in \dom(\pi \circ [X_a])$.  Then $\pi([X_a](t))=\st(h^{[t \sigma]})=1$, whence $[X_a]\in L(H)$.  Since $H$ is totally disconnected, $L(H)$ is trivial and hence $[X_a]=\O$, a contradiction.  Thus we have $\sigma=(r+\epsilon)\nu$ for some $r\in \r^{>0}$ and infinitesimal $\epsilon \in \r^*$.  Thus $\pi \circ [X_a]=r\y$ and $\x:=\frac{1}{r}[X_a]$ is the desired lift of $\y$. 
\end{proof}

\

\noindent We need one fact from the global setting that we include here for completeness.  It is taken from ~\cite{H}.

\

\begin{lemma}\label{L:hirsch}
Suppose $G$ is a pure topological group such that there are no nontrivial 1-parameter subgroups $X:\r \rightarrow G$.  Then $G$ has a neighborhood base at $1$ of open subgroups of $G$.  In particular, $G$ is totally disconnected.
\end{lemma}

\

\noindent Since locally euclidean local groups are locally connected, the next theorem, in combination with Lemma \ref{L:eucNSCS}, finishes the proof of the Local H5.

\

\begin{thm}\label{T:eucNSS}
If $G$ is locally connected and $\Nsc$, then $G$ is $\Ns$.
\end{thm}

\begin{proof}
Let $V$ be a compact symmetric neighborhood of $1$ in $G$ such that $V$ contains no nontrivial connected subgroups.  Choose an open neighborhood $W$ of $1$ in $G$ with $W\subseteq V$ and a compact subgroup $H_1$ of $G$ with $H_1\subseteq V$ with the property that every subgroup of $G$ contained in $W$ is contained in $H_1$; these are guaranteed to exist by Lemma \ref{L:neighandgroup}.  Without loss of generality, we can assume $W\subseteq \u_6$.  We now observe that if $a\in \boldsymbol{\mu}$ is degenerate, then $a\in H_1^*$, for the internal subgroup internally generated by $a$ is contained in $\boldsymbol{\mu}\subseteq W^*$.  Since $G$ is pure, if $a\in \boldsymbol{\mu}$ is nondegenerate, then it is pure; for such an $a$, if $a\in H_1^*$, then $a^\nu\in H_1^*\subseteq V^*$ for all $\nu$, contradicting the fact that some power of $a$ lives outside of $V$.  Hence, nondegenerate infinitesimals live outside of $H_1^*$.  

\

\noindent Since $H_1\subseteq V$, $H_1$ admits no nontrivial 1-parameter subgroups, and hence $H_1$ must be totally disconnected by Lemma \ref{L:hirsch}.  Choose an open (in $H_1$) subgroup $H$ of $H_1$ contained in $H_1\cap W$.  Write $H=H_1\cap W_1$ with $W_1$ an open neighborhood of $1$ in $G$ and $W_1\subseteq W$.  Since $H$ is an open subgroup of $H_1$, it is also closed in $H_1$, and thus $H$ is a compact subset of $G$.  This implies that the set $U:=\{a\in W_1 \ | \ aHa^{-1}\subseteq W_1\}$ is open.  Fix $a\in U$.  Since $aHa^{-1}$ is a subgroup of $G$ contained in $W_1\subseteq W$, we have $aHa^{-1}\subseteq H_1$.  Thus $aHa^{-1}\subseteq H_1\cap W_1=H$ implying that $H$, considered as a sublocal group of $G$, is normal.  Replacing $U$ by $U\cap U^{-1}$, we can take $U$ as the associated normalizing neighborhood for $H$.

\

\noindent Note that $H\subseteq \u_6$ and $H^6\subseteq U$.  Find symmetric open neighborhoods $W_2$ and $W_3$ of $1$ in $G$ so that $H\subseteq W_3 \subseteq \overline{W_3} \subseteq W_2$ and so that $W_2 \subseteq \u_6$ and $W_2^6 \subseteq U$; this is possible by the regularity of $G$ as a topological space and the fact that $H$ is compact.  Let $G':=(G/H)_{W_2}$.  Suppose $a\in \boldsymbol{\mu}$.  If $a$ is degenerate, then we have that $a\in H_1^* \cap W_1^* = H^*$, whence $\pi(a)=1_{G'}$.  Now suppose that $a$ is nondegenerate.  Let $\tau:=\ord_{\overline{W_3}}(a)$.  Then $a^\tau \in (\overline{W_3})^*$, but $a^{\tau+1}\in W_2^* \setminus (\overline{W_3})^*$.  Choose an open neighborhood $U'$ of $1$ in $G$ so that $U'H\subseteq W_3$.  Suppose $\pi(a^{\tau+1})=\pi(x)$ for some $x\in (U')^*$.  Then $a^{\tau+1}=xh$ for some $h\in H^*$, whence $a^{\tau+1}\in W_3^*$, a contradiction.  Hence $\pi(a)^{\tau+1}\notin \pi((U')^*)$.  Meanwhile, for $i=\lilo(\tau)$, we have $\pi(a)^i\in \boldsymbol{\mu}(G')$.  Thus $\pi(a)$ is pure in $G'$ and we have shown that $G/H$ has no degenerate infinitesimals other than $1_{G'}$; in other words, we have shown that $G'$ is $\Ns$.

\

\noindent  Since $G$ and $G'$ are pure, $L(\pi):L(G)\rightarrow L(G')$ is a continuous $\r$-linear map.  Since $H$ is totally disconnected, $L(\pi)$ is surjective by Lemma \ref{L:lift}.  Note that if $\x \in \ker(L(\pi))$, then $\x$ can be considered a local 1-ps of $H$.  Since $H$ is totally disconnected, $L(H)$ is trivial, and hence $\x=\O$.  Since $L(G')$ is finite-dimensional, we can conclude that the map $L(\pi)$ is an isomorphism of real topological vector spaces.

\

\noindent Now take a special neighborhood $\u'$ of $G'$ and let $E':\k'\rightarrow K'$ denote the local exponential map for $G'$.  Take a connected neighborhood $\u$ of $1$ in $G$ such that $\u\subseteq W_2$ and $\pi(\u)\subseteq K'$.  Let $x\in \u$.  Since $E'$ is a bijection, there is a unique $\y \in \k'$ such that $\pi(x)=E(\y)$.  Since $L(\pi)$ is a bijection, there is a unique $\x\in L(G)$ such that $\pi\circ \x=\y$.  Thus we can write $x=\x(1)x(H)$ where $x(H)\in H$.  One can easily verify that the map which assigns to each $x\in \u$ the above $\y \in L(G')$ is continuous.  From this and the fact that $L(\pi)$ is a homeomorphism, we see that the map $x\mapsto x(H):\u \rightarrow H$ is continuous.  Since $\u$ is connected, $H$ is totally disconnected, and $1(H)=1$, it follows that $x(H)=1$ for all $x\in \u$.  Now by the injectivity of $L(\pi)$ and $E'$, we have that $\pi|\u$ is injective, implying that $G$ cannot have any subgroups other than $\{1\}$ contained in $\u$.  We thus conclude that $G$ is $\Ns$.
\end{proof}

\end{document}